\UseRawInputEncoding

\RequirePackage[l2tabu, orthodox]{nag}
\documentclass[10pt,a4paper,reqno,oneside,final, english]{smfart}

\usepackage{leftidx}
\usepackage[titletoc]{appendix}
\usepackage{tikz}\usetikzlibrary{decorations.markings,matrix,arrows}
\usepackage{tikz-cd}
\usepackage{amsfonts}
\usepackage{amsthm}
\usepackage[T1]{fontenc}
\usepackage[mathscr]{eucal}
\usepackage{setspace}
\usepackage{amssymb,latexsym,amsmath,amscd}
\usepackage{graphicx}
\usepackage{showkeys}
\usepackage{layout}
\usepackage{enumerate}
\usepackage{indentfirst}
\usepackage[varg]{pxfonts}
\usepackage[stretch=10]{microtype}
\usepackage{booktabs}
\usepackage{xspace}
\usepackage{calrsfs}
\usepackage{filecontents}
\usepackage{mathtools}
\usepackage{titlesec}
\usepackage{fixltx2e}
\usepackage{todonotes}
\usepackage{scalerel,stackengine}
\usepackage{enumitem}
\usepackage[backref = page, colorlinks   = true,
     citecolor    = gray, pdfborder={0 0 0}]{hyperref}
\usepackage{nameref}
\usepackage{cleveref}

\allowdisplaybreaks

\titleformat{\paragraph}
{\normalfont\normalsize\bfseries}{\theparagraph}{1em}{}
\titlespacing*{\paragraph}
{0pt}{3.25ex plus 1ex minus .2ex}{1.5ex plus .2ex}

\makeatletter
\newcommand{\neutralize}[1]{\expandafter\let\csname c@#1\endcsname\count@}
\makeatother

\theoremstyle{plain}
\newtheorem{thmm}{Theorem}[section]

\newtheorem{thm}[thmm]{Theorem}
\newtheorem*{thm*}{Theorem}

\newtheorem{lem}[thmm]{Lemma}

\newtheorem{lem-def}[thmm]{Lemma-Definition}
\newtheorem{claim}[thmm]{Claim}
\newtheorem{pro}[thmm]{Proposition}

\newtheorem{pro-def}[thmm]{Proposition-Definition}
\newtheorem{cor}[thmm]{Corollary}

\newtheorem{que}[thmm]{Question}

\theoremstyle{definition}
\newtheorem{Def}[thmm]{Definition}

\newtheorem{rem}[thmm]{Remark}

\theoremstyle{remark}

\newcommand{\ssec}{\subsection}

\newcommand{\ti}[1]{\tilde{#1}}
\newcommand{\ul}{\underline}
\newcommand{\vast}{\bBigg@{4}}
\newcommand{\Vast}{\bBigg@{5}}
\newcommand{\wt}{\widetilde}

\stackMath
\newcommand\reallywidehat[1]{%
\savestack{\tmpbox}{\stretchto{%
  \scaleto{%
    \scalerel*[\widthof{\ensuremath{#1}}]{\kern-.6pt\bigwedge\kern-.6pt}%
    {\rule[-\textheight/2]{1ex}{\textheight}}
  }{\textheight}%
}{0.5ex}}%
\stackon[1pt]{#1}{\tmpbox}%
}
\definecolor{armygreen}{rgb}{0.29, 0.33, 0.13}
\definecolor{ao(english)}{rgb}{0.5, 0.2, 0.0}

\newcommand{\bC}{\mathbf{C}}

\newcommand{\bH}{\mathbf{H}}

\newcommand{\bQ}{\mathbf{Q}}
\newcommand{\bR}{\mathbf{R}}

\newcommand{\bV}{\mathbf{V}}

\newcommand{\bZ}{\mathbf{Z}}

\newcommand{\bm}{\mathbf{m}}

\newcommand{\cA}{\mathcal{A}}

\newcommand{\cB}{\mathcal{B}}

\newcommand{\cE}{\mathcal{E}}
\newcommand{\cF}{\mathcal{F}}

\newcommand{\cH}{\mathcal{H}}

\newcommand{\cJ}{\mathcal{J}}

\newcommand{\cO}{\mathcal{O}}

\newcommand{\cU}{\mathcal{U}}
\newcommand{\cV}{\mathcal{V}}

\newcommand{\cX}{\mathcal{X}}

\newcommand{\fU}{\mathfrak{U}}

\newcommand{\gD}{\Delta}

\newcommand{\gO}{\Omega}
\newcommand{\gS}{\Sigma}

\newcommand{\ga}{\alpha}
\newcommand{\gb}{\beta}
\newcommand{\gd}{\delta}

\newcommand{\gl}{\lambda}

\newcommand{\gs}{\sigma}
\newcommand{\gt}{\theta}

\newcommand{\Alb}{\mathrm{Alb}}

\newcommand{\coker}{\mathrm{coker}}
\newcommand{\Cone}{\mathrm{Cone}}

\newcommand{\Gal}{\mathrm{Gal}}

\newcommand{\Id}{\mathrm{Id}}
\newcommand{\Ima}{\mathrm{Im}}

\newcommand{\lcm}{\mathrm{lcm}}

\newcommand{\pr}{\mathrm{pr}}

\newcommand{\sing}{\mathrm{sing}}

\newcommand{\tors}{\mathrm{tors}}

\newcommand{\tr}{\mathrm{tr}}

\newcommand{\cf}{\emph{cf.} }
\newcommand{\eg}{\emph{e.g.} }

\newcommand{\bss}{\backslash}

\newcommand{\colonec}{\mathrel{:=}}
\newcommand{\cnec}{\mathrel{:=}}

\newcommand{\st}{\star}

\renewcommand{\(}{\left(}
\renewcommand{\)}{\right)}

\newcommand{\dto}{\dashrightarrow}

\newcommand*\eto{%
  \xrightarrow[]{\raisebox{-0.25 em}{\smash{\ensuremath{\sim}}}}%
}
\newcommand{\hto}{\hookrightarrow}
\newcommand{\xto}[1]{\xrightarrow{ #1 }}

\makeatletter
\let\orgdescriptionlabel\descriptionlabel
\renewcommand*{\descriptionlabel}[1]{%
  \let\orglabel\label
  \let\label\@gobble
  \phantomsection
  \edef\@currentlabel{#1}%
  \let\label\orglabel
  \orgdescriptionlabel{#1}%
}
\makeatother
\tikzset{node distance=2cm, auto}

\numberwithin{equation}{section}

\title{Algebraic approximations of fibrations in abelian varieties over a curve} 

\author{Hsueh-Yung Lin}
\address{Department of Mathematics, National Taiwan University, 
	No. 1, Sec. 4, Roosevelt Rd., Taipei 10617, Taiwan.}
\email{hsuehyunglin@ntu.edu.tw}

\begin{document}

\begin{abstract}
For every fibration $f : X \to B$ with $X$ a compact K\"ahler manifold, $B$ a smooth projective curve, and a general fiber of $f$ an abelian variety, we prove that $f$ has an algebraic approximation.
\end{abstract}

\maketitle

\section{Introduction}\label{sec-intro}

Let $X$ be a compact K\"ahler manifold. An \emph{algebraic approximation of $X$} is a deformation $\cX \to \gD$ of $X$ such that, up to shrinking $\gD$, the subset parameterizing projective manifolds in this family is dense in $\gD$. Whether or not a compact K\"ahler manifold admits an algebraic approximation is generally known as the Kodaira problem.  For compact K\"ahler surfaces, Kodaira showed that algebraic approximations always exist~\cite{KodairaSurfaceII}. But starting from dimension 4 and on, there exist compact K\"ahler manifolds constructed by Voisin which do not have any algebraic approximation~\cite{Voisincs}.

For most compact K\"ahler manifolds, the existence of algebraic approximations is still unknown. Non-trivial examples of compact K\"ahler manifolds admitting algebraic approximations can be found in~\cite{Schrackdefo, CaoJApproxalg, GrafDefKod0, ClaudonToridefequiv, ClaudonHorpi1, HYLkodfibellip} and the list is rather exhaustive at present. This list includes especially elliptic fibrations~\cite{HYLkodfibellip} and smooth fibrations in abelian varieties~\cite{ClaudonToridefequiv} (both over a projective manifold), and they therefore motivate the following question, which we will study in this text.

\begin{que}\label{que-appalgfibab}
Let $f : X \to B$ be a fibration whose general fiber is an abelian variety. Assume that $X$ is a compact K\"ahler manifold  and $B$ a projective manifold, does $f$ have an algebraic approximation?
\end{que}

In the case where $f$ is smooth, some algebraic approximation of $f$ can be realized by the so-called \emph{tautological family} associated to $f$~\cite{ClaudonToridefequiv}, which is described as follows. Let $J \to B$ denote the Jacobian fibration associated to $f$ and $\cJ$ its sheaf of germs of local holomorphic sections. There is a one-to-one correspondence between the set of isomorphism classes of $J$-torsors (resp. projective $J$-torsors) and $H^1(B,\cJ)$  (resp. $H^1(B,\cJ)_{\tors}$), and the tautological family associated to $f$ is a family of $J$-torsors
$$
\ \ \ \ \ \ \ \ \ \ \ \ \ \ \ \ \ \ 
\Pi : \cX \xto{q}  B \times V \to V \cnec H^1(B, R^{g}f_*\gO^{g-1}_{X/B}) \ \ \ \ \ \ \ \ (g \cnec \dim X - \dim B)
$$
  such that $t \in V$ parameterizes the $J$-torsor represented by
$$\eta(f) + \exp(t) \in H^1(B,\cJ)$$
where $\eta(f) \in H^1(B,\cJ)$ is the element associated to $f$ and 
$$\exp : V \to H^1(B,\cJ)$$ 
is the map induced by the quotient $R^{g}f_*\gO^{g-1}_{X/B} \to \cJ$. An argument involving Hodge theory shows that $V$ contains a dense subset parameterizing projective $J$-torsors~\cite{ClaudonToridefequiv}.

When $f: X \to B$ is an elliptic fibration, the situation is much more complicated but the construction of an algebraic approximation of $f$ in~\cite{HYLkodfibellip} is still based on the tautological family associated to some specific elliptic fibration and Hodge theory is again involved in the proof of the density result. This raises the question of whether we can construct, more generally, for every fibration whose general fiber is an abelian variety, the analogue of tautological family and prove (by Hodge-theoretic arguments) that projective members are dense in this family, thereby answering Question~\ref{que-appalgfibab}.

In this text, we will focus on fibrations $f : X \to B$ with $\dim B = 1$ and show that the above approach can indeed be carried out  in this situation.  Due to the presence of singular fibers,  the proof is technically more involved than in the smooth case and the implementation of the idea relies on work of Deligne, El Zein, and Zucker on variations of Hodge structures over a curve~\cite{ZuckerHdgL2, ElZeinZuck} and M. Saito's compactifications of Jacobian fibrations~\cite{MSaitoAdmnorm}. We will first show that under the assumption that $f$ has local sections at every point of $B$, 
we can construct a tautological family associated to $f$
 (see Proposition-Definition~\ref{pro-def-existfam}) and prove that it is an algebraic approximation.

\begin{thm}\label{thm-AbFibDefprec} 
Let $G$ be a finite group and let $f : X \to B$ be a $G$-equivariant fibration 
with $X$ being a compact K\"ahler manifold, $B$ a smooth projective curve, and a general fiber of $f$ an abelian variety. Assume that $f$ has local sections at every point of $B$, then the $G$-equivariant tautological family 
$$\Pi :  \cX \to B \times V \to V  $$
associated to $f$ is an algebraic approximation of $f$
which is $G$-equivariantly locally trivial over $B$ (see Definition~\ref{def-loctriv}). 
\end{thm}

Without assuming that the fibration $f : X \to B$ has local sections everywhere in $B$, the existence of an algebraic approximation of $f$ will follow as a corollary of Theorem~\ref{thm-AbFibDefprec}, which answers Question~\ref{que-appalgfibab} in the affirmative when $\dim B = 1$.

\begin{cor}\label{cor-AbFibDef} 
Let $f : X \to B$ be a fibration with $X$ being a compact K\"ahler variety, $B$ a smooth projective curve, and a general fiber of $f$ an abelian variety. Then $f$ has an algebraic approximation 
which is locally trivial over $B$.
\end{cor}

Our original motivation of this work comes from the long-standing Kodaira problem in dimension 3 which we will study in~\cite{HYLkod3}. Roughly speaking, given a non-uniruled compact K\"ahler threefold $X$ of algebraic dimension $a(X) = 1$, the algebraic reduction of $X$ is bimeromorphic to a fibration $f: X' \to B$ such that either $f$ is isotrivial or a general fiber of $f$ is an abelian surface~\cite{FujikiStruC}. In the latter case, the existence of algebraic approximations of $X$ will be a consequence of Corollary~\ref{cor-AbFibDef}.  The reader is referred to~\cite{HYLkod3} for the detail.

In order to explain the basic idea of the proof of Theorem~\ref{thm-AbFibDefprec} without dealing with technical difficulties,  we will first recall (following~\cite{ClaudonToridefequiv}) in Section~\ref{sec-fiblisse} the construction of the tautological family $\Pi$ associated to a smooth torus fibration $f$, together with a sketch of proof of the fact that $\Pi$ is an algebraic approximation under the assumption that the base and the fibers of $f$ are projective~\cite[Theorem 1.2]{ClaudonToridefequiv}.  We will then construct in Section~\ref{sec-genjac} the tautological families associated to fibrations as in Theorem~\ref{thm-AbFibDefprec},  and prove Theorem~\ref{thm-AbFibDefprec} and Corollary~\ref{cor-AbFibDef} in Section~\ref{sec-appalgtau}.

\section{Preliminaries and general results}\label{sec-prelim}

\ssec{Basic notions and terminologies}
\hfill

In this text, a \emph{fibration} is a proper holomorphic surjective map $f : X \to B$ whose general fiber is irreducible. The fiber $f^{-1}(b)$ of $f$ over $b \in B$ will often be denoted by $X_b$. A \emph{local section} of a fibration $f$ at a point $b \in B$ is a section of $f^{-1}(U) \to U$ for some (Euclidean) neighborhood $U \subset B$ of $b$. A \emph{multi-section} of $f$ is a subvariety $\gS \subset X$ such that $f_{|\gS}$ is surjective, finite, and flat over $B$. 

A deformation of a compact complex variety $X$ is a 
surjective, proper, and flat morphism $\Pi : \cX \to \gD$ containing $X$ as a fiber. 
Let $f : X \to B$ be a holomorphic map 
from a compact complex variety $X$.  
A \emph{deformation of $f$} (with fixed target) is a family of maps $f_t : \cX_t \to B$ containing $f$ as a member. Namely, it is a composition 
$$\Pi : \cX \xto{q} B \times \gD \xto{\pr_2} \gD$$ 
such that  $\Pi$ is a deformation of $X$ and $q_{|\cX_o} : \cX_o \to B$ equals $f$ for some $o \in \gD$. Note that in the definition, the target of $f_t$ does not deform in the family.

Let $G$ be a group and 
$X$ a compact complex variety endowed with a $G$-action. 
We say that a deformation $\Pi : \cX \to \gD$ of $X$ \emph{preserves the $G$-action}, or $\Pi$ is a \emph{$G$-equivariant deformation of $X$}, if there exists a $G$-action on $\cX$ extending the given $G$-action on $X$ such that $\Pi$ is $G$-invariant. Similarly, let $f : X \to B$ be a $G$-equivariant map. We say that a deformation $\Pi : \cX \xto{q}  B \times \gD  \to \gD$ of $f$ \emph{preserves the $G$-action} (or is \emph{$G$-equivariant}) if there exists a $G$-action on $\cX$ extending the $G$-action on $f$ such that $q$ is $G$-equivariant (the $G$-action on $B \times \gD$ being the pullback of the given $G$-action on $B$) and $\Pi : \cX \to \gD$ is $G$-invariant.

\begin{Def}[Locally trivial deformations]\label{def-loctriv}
\hfill
\begin{enumerate}[label = \roman{enumi})]
\item\label{def-loctriv2} A deformation $\Pi : \cX \xto{q} B \times \gD \to \gD$ of $f : X \to B$ is called \emph{locally trivial over $B$} if there exists an open cover $\{U_i\}$ of $B$ such that $q^{-1}(U_i \times \gD) \simeq f^{-1}(U_i) \times \gD$ over $U_i \times \gD$ for every $i$.
\item In i), let $G$ be a group and $f : X \to B$ a $G$-equivariant map. We say that $\Pi$ is \emph{$G$-equivariantly locally trivial over $B$} if $\Pi$ preserves the $G$-action and  the isomorphisms $q^{-1}(U_i \times \gD) \simeq f^{-1}(U_i) \times \gD$ above are $G$-equivariant for some $G$-invariant open cover $\{U_i\}$ of $B$.
\end{enumerate}
\end{Def}

An obvious property about locally trivial deformations is that the quotient of a $G$-equivariantly locally trivial deformation is still  locally trivial.

\begin{lem}\label{lem-Gquotloctriv}
	Let $G$ be a finite group and let $f : X \to B$ be a $G$-equivariant fibration.
If $\Pi : \cX \xto{q} B \times \gD \to \gD$ is a  deformation of  
$f : X \to B$ which is $G$-equivariantly  locally trivial over $B$, then the quotient 
$$\Pi' : \cX/G \xto{q'} (B/G) \times \gD \to \gD$$ 
of $\Pi$ by $G$ is a deformation of $f' : X/G \to B/G$ which is locally trivial over $B/G$.
\end{lem}
\begin{proof}

By assumption, there exists a $G$-invariant open cover $\{U_i\}$ of $B$ such that $q^{-1}(U_i \times \gD)$ is $G$-equivariantly isomorphic to $f^{-1}(U_i) \times \gD$ over $U_i \times \gD$. Therefore if $U'_i \subset B/G$ denotes the image of $U_i$ in $B/G$, then $\{U_i'\}$ is an open cover of $B/G$ and we have 
an isomorphism
$$q'^{-1}(U_i' \times \gD) = q^{-1}(U_i \times \gD)/G \simeq (f^{-1}(U_i)  \times \gD) / G = f'^{-1}(U'_i)  \times \gD,$$
over $U'_i \times \gD$
where, by abuse of notation, $q^{-1}(U_i \times \gD)/G$ denotes the image of $q^{-1}(U_i \times \gD)$ in $\cX/G$ and same for $(f^{-1}(U_i)  \times \gD) / G$.
\end{proof}%

Now we come to the notion of algebraic approximation. Recall that a compact complex variety $X$ is called Moishezon if its algebraic dimension $a(X)$ is equal to $\dim X$.

\begin{Def}[Algebraic approximation]\label{def-appalg}
	Let $X$ be a compact complex variety. An algebraic approximation of $X$ is a deformation $\Pi : \cX \to \gD$ of $X$ such that up to shrinking $\gD$, the subset of points in $\gD$ parameterizing Moishezon varieties is dense for the Euclidean topology.  
\end{Def}

When $X$ is a compact K\"ahler manifold, Definition~\ref{def-appalg} 
coincides with the definition given
at the beginning of the introduction:
this follows from Moishezon's projectivity criterion, 
together with the fact that
small deformations of a compact K\"ahler manifold 
remain compact K\"ahler manifolds.

\ssec{Equivariant K\"ahler desingularizations}

\begin{Def}\label{def-mindesingG}
	Let $X$ be a compact K\"ahler variety endowed with a $G$-action for some group $G$. 
	A \emph{strong $G$-equivariant K\"ahler desingularization} of $X$ 
	is a bimeromorphic holomorphic map $\nu : \ti{X} \to X$ satisfying the following properties
	\begin{enumerate}[label = \roman{enumi})]
		\item $\ti{X}$ is a compact K\"ahler manifold.
		\item $\nu : \ti{X} \to X$ is projective. 
		Moreover, if $X_{\sing} \subset X$ denotes the singular locus, 
		then the restriction 
		$\ti{X} \bss\nu^{-1}(X_{\sing}) \to X \bss X_{\sing}$ of $\nu$ is an isomorphism.
		\item There exists a $G$-action on $\ti{X}$ such that $\nu$ is $G$-equivariant.
	\end{enumerate}
\end{Def}

As an immediate consequence of the existence of functorial resolution of singularities
(see \eg~\cite[Theorem 3.45]{resing} 
and~\cite[Chapter 3]{resing} for a historical account of this result starting from 
Hironaka's desingularization; see also~\cite{MR2500573}
for the existence of desingularizations of complex analytic spaces),
strong equivariant K\"ahler desingularizations always exist.

\begin{thm}\label{thm-mindesingG}
	Let $X$ be a compact K\"ahler variety endowed with a $G$-action. 
	A strong $G$-equivariant K\"ahler desingularization of $X$ always exists.
\end{thm}

\begin{proof}
	The existence of a desingularization $\nu:  \ti{X} \to X$ of $X$
	satisfying ii) and iii) in Definition~\ref{def-mindesingG}
	follows from the existence of 
	functorial resolution of singularities (see \eg~\cite[Theorem 3.45]{resing}).
	Since projective morphisms are K\"ahler~\cite[Lemma 4.4.(1)]{FujikiClosednessDouady},
	it follows from~\cite[Proposition II.1.3.1.(v) and (vi)]{VarouchasKS} 
	that $\ti{X}$ is K\"ahler.
\end{proof}

\ssec{Campana's criterion}\hfill

Let $X$ be a complex variety. We say that $X$ is \emph{algebraically connected} if a general pair of points $x,y \in X$ is contained in a compact connected (but not necessarily irreducible) curve of $X$. We have the following criterion proven by Campana for a variety to be Moishezon in terms of algebraic connectedness. 

\begin{thm}[Campana {\cite[Corollaire on p.212]{CampanaCored}}]\label{thm-algconn}
Let $X$ be a compact complex variety bimeromorphic to a compact K\"ahler manifold. Then $X$ is Moishezon if and only if $X$ is algebraically connected. 
\end{thm}%

Together with Moishezon's criterion, Theorem~\ref{thm-algconn} implies that a compact complex manifold $X$ is projective if and only if $X$ is K\"ahler and algebraically connected.
Since we will mainly deal with fibrations $f : X \to B$ with  $\dim B = 1$, here is a variant of Campana's criterion in this particular situation.

\begin{cor}[Special case of Campana's criterion]\label{cor-multsecMoibase} 
Let $f : X \to B$ be a fibration from a compact K\"ahler variety (resp. manifold) 
to a smooth projective curve. 
Assume that a general fiber of $f$ is Moishezon (resp. projective), 
then $X$ is Moishezon (resp. projective) if and only if $f$ has a multi-section.
\end{cor}%

\section{Smooth torus fibrations and their tautological families}\label{sec-fiblisse}

The heuristic proving Theorem~\ref{thm-AbFibDefprec} is originated from the smooth case~\cite{ClaudonToridefequiv}, although the basic idea can be traced back to Kodaira~\cite{KodairaSurfaceII} in his studies of deformations of elliptic surfaces. So in this section, we will recall the construction of the tautological families associated to smooth torus fibrations and explain how (assuming the fibers and the bases of the fibrations are projective) these families are proven to be algebraic approximations following~\cite{ClaudonToridefequiv}.

Let $f : X \to B$ be a smooth torus fibration of relative dimension $g$ and  $J \to B$  the Jacobian fibration associated to $f$. The sheaf $\cJ_{\bH/B}$ of germs of holomorphic sections of $J \to B$ lies in the exact sequence
\begin{equation}\label{SE-Jac}
\begin{tikzcd}[cramped, row sep = 5, column sep = 40]
0 \ar[r] & \bH  \arrow[r]  & \cE_{\bH/B} \ar[r,"\exp"] & \cJ_{\bH/B} \ar[r] & 0 
\end{tikzcd}
\end{equation}
where $\bH \colonec R^{2g-1}f_*\bZ$ and 
$$\cE_{\bH/B} \colonec (\bH   \otimes \cO_{B}) / R^{g-1}f_*\gO^g_{X/B} \simeq  R^gf_*\gO^{g-1}_{X/B}.$$ 
Every morphism $\phi : B' \to B$ induces a map $\cJ_{\bH/B} \to  \phi_*\cJ_{\phi^{-1}\bH/B'}$ by pulling back sections.

The fibration $f$ is a $J$-torsor and to each isomorphism class of $J$-torsors, we can associate in a biunivocal way an element $\eta(f) \in H^1(B,\cJ_{\bH/B})$ satisfying the property that $\eta(f)$ is torsion if and only if $f$ has a multi-section~\cite[Proposition 2.2]{ClaudonToridefequiv}. Moreover, if  
$$\exp : H^1(B ,\cE_{\bH/B}) \to H^1(B, \cJ_{\bH/B})$$
 denotes the morphism induced by $\exp : \cE_{\bH/B} \to \cJ_{\bH/B}$, then there exists a family
\begin{equation}\label{fam-Jtors}
\Pi : \cX \xto{q}  B \times V \to V \cnec H^1(B ,\cE_{\bH/B})
\end{equation}
of $J$-torsors such that  $t \in V$ parameterizes the $J$-torsor which corresponds to 
$$\eta(f) + \exp(t) \in H^1(B, \cJ_{\bH/B}).$$ 
 The family $\Pi$ is called the \emph{tautological family} associated to $f$.

Concretely, $\Pi$ is constructed as follows. Let  $\pr_1 : B \times V \to B$ be the first projection and let
$$\xi \in H^1(B, \cE_{\bH/B}) \otimes V^\vee \subset H^1(B, \cE_{\bH/B}) \otimes H^0(V,\cO_V) \simeq  H^1(B \times V, \cE_{\pr_1^{-1}\bH/B \times V})$$
 be the element which corresponds to the identity $\Id : V \to H^1(B, \cE_{\bH/B})$. Let 
$$\gt \cnec \pr_1^*\eta(f) + \wt{\exp}(\xi) \in  H^1(B \times V, \cJ_{\pr_1^{-1}\bH/B \times V})$$
where $\pr_1^* : H^1(B, \cJ_{\bH/B}) \to  H^1(B \times V, \cJ_{\pr_1^{-1}\bH/B \times V})$ is the map induced by $\cJ_{\bH/B} \to {\pr_1}_*\cJ_{\pr_1^{-1}\bH/B \times V}$ and 
$$ \wt{\exp} : H^1(B, \cE_{\pr_1^{-1}\bH/B \times V}) \to H^1(B \times V, \cJ_{\pr_1^{-1}\bH/B \times V})$$
the map induced by $\exp : \cE_{\pr_1^{-1}\bH/B \times V} \to \cJ_{\pr_1^{-1}\bH/B \times V}$. Then the smooth torus fibration $q : \cX \to B \times V$ defining $\Pi$ is defined to be the $(J \times B)$-torsor corresponding to $\gt$. As $\gt$ can be represented by a \v{C}ech 1-cocycle $(\gt_{ij})$ with respect to an open cover of $B \times V$ of the form $\{U_i \times V\}$ where $\{U_i \}$ is a good open cover of $B$, it follows that $\Pi$ is locally trivial over $B$.

In the case where $f : X \to B$ is a $G$-equivariant smooth torus fibration for some finite group $G$, there is a natural $G$-action on~\eqref{SE-Jac}. The sub-family of~\eqref{fam-Jtors} over $V^G$ is a deformation of $f : X \to B$ preserving the $G$-action~\cite[Proposition 2.10]{ClaudonToridefequiv}, called the \emph{$G$-equivariant tautological family} associated to $f$.

\begin{thm}[{Claudon~\cite{ClaudonToridefequiv}}]\label{thm-multsec}
Let $G$ be a finite group and $f : X \to B$ a $G$-equivariant smooth torus fibration. Assume that the total space $X$ is a compact K\"ahler manifold. Then in the $G$-equivariant tautological family $\Pi$ associated to $f$, the subset of $V^G \cnec H^1(B ,\cE_{\bH/B})^G$ parameterizing fibrations with a multi-section is dense in $V^G$. In particular, if the fibers of $f$ and $B$ are projective, then $\Pi$ is an algebraic approximation of $f$.
\end{thm}

\begin{proof}[Sketch of proof of Theorem~\ref{thm-multsec}]

By Deligne's theorem, $H \colonec H^1(B, \bH)$ is a pure Hodge structure of degree $2g$ (where $g$ is the relative dimension of $f$) satisfying the Hodge symmetry and concentrated in bi-degrees $(g-1,g+1)$, $(g,g)$, and $(g+1,g-1)$~\cite[Section 2]{ZuckerHdgL2}. Also, if $F^\bullet H_\bC$ denotes the Hodge filtration of $H$, then $V \cnec H^1(B ,\cE_{\bH/B})$ is isomorphic to $H_\bC/F^gH_\bC$~\cite[Section 2]{ZuckerHdgL2}. It follows that the composition  
$$\mu : H_\bR \hto H_\bC  \to V.$$
is surjective, so $\mu(H_{\bQ})$ is dense in $V$. Since $G$ is finite, we have 
$$\mu(H_{\bQ}^G) \otimes \bR = \mu(H_{\bQ})^G \otimes \bR = V^G.$$
Therefore $\mu(H_{\bQ}^G) $ is dense in $V^G$.

The K\"ahler assumption of $X$ implies that the image of the $G$-equivariant class $\eta_G(f) \in H^1_G(B,\cJ)$ associated to $X$ (which is a refinement of $\eta(f)$, see~\cite[Section 2.4]{ClaudonToridefequiv}) under the connecting morphism
$$H^1_G(B,\cJ) \to H^2_G(B,\bH)$$
induced by~\eqref{SE-Jac} is torsion~\cite[Proposition 2.11]{ClaudonToridefequiv}. So there exist $m \in \bZ_{>0}$ and $t_0 \in V^G$ such that $m\eta(f) = \exp(t_0)$. Therefore $\eta(f) + \exp\(t - \frac{t_0}{m}\)$ is torsion for every $t \in \mu(H_{\bQ}^G)$, so each of the fibrations $\cX_t \to B$ parameterized by the subset 
$$\mu(H_{\bQ}^G) - \frac{v_0}{m} \subset V^G$$
in the tautological family~\eqref{fam-Jtors} has an étale multi-section over $B$~\cite[Proposition 2.2]{ClaudonToridefequiv}. 
As $\mu(H_{\bQ}^G)$ is dense in $V^G$, so is $\mu(H_{\bQ}^G) - \frac{v_0}{m}$. 
Hence $V^G$ contains a dense subset parameterizing fibrations in the tautological family having a multi-section over $B$. The last statement follows from the main statement of Theorem~\ref{thm-multsec} and Corollary~\ref{cor-multsecMoibase}.
\end{proof}%

\section{Fibrations in abelian varieties over a curve}\label{sec-genjac}

In this section, we study deformations of (\emph{a priori} non-smooth) fibrations in abelian varieties over a curve in a similar spirit of what we did in Section~\ref{sec-fiblisse}.  In \S\ref{ssec-thmZucker} and \S\ref{ssec-genjac}, we first focus on the Hodge-theoretic ingredients we need in the proof of Theorem~\ref{thm-AbFibDefprec} based on Zucker's theory of variations of Hodge structures over a curve~\cite{ZuckerHdgL2} and also~\cite{ElZeinZuck}. Then we introduce and study in \S\ref{ssec-bimJtors} the notion of \emph{bimeromorphic $J$-torsors} and construct the tautological families associated to them.

\ssec{Variation of Hodge structures over a curve: Zucker's theorem}\label{ssec-thmZucker}
\hfill

First we recall Zucker's result on the variations of Hodge structures (VHS) over a curve~\cite{ZuckerHdgL2}.  Let $B^\st$ be a smooth quasi-projective curve and $\bH$ an (integral) local system which underlies a VHS of weight $m$ over $B^\st$.  Let $j : B^\st \hto {B}$ be the smooth compactification of $B^\st$. 

\begin{thm}[Zucker~{\cite[Theorem 7.12]{ZuckerHdgL2}}]\label{thm-ZuckSH}
Assume that $\bH$ underlies an $\bR$-polarized VHS of weight $m$ whose local monodromies 
around $B \bss B^\st$ are quasi-unipotent. Then $H^i({B} , j_* \bH)$ has a natural Hodge structure of weight $m+i$. Moreover, if $\bH$ satisfies the Hodge symmetry, then $H^i({B} , j_* \bH)$  satisfies the Hodge symmetry as well.
\end{thm}

Here an \emph{$\bR$-polarized} VHS is a VHS admitting a flat bilinear pairing defined over $\bR$ on the Hodge bundle $\bH_\bC = \bH \otimes \bC$  whose restriction to each fiber of $\bH_\bC$ is a polarization of the underlying Hodge structure. This is the case when $\bH = R^mf^\st_*\bZ$ where $f^\st : X^\st \to B^\st$ is a family of compact K\"ahler manifolds. The quasi-unipotence of the monodromy action is satisfied for instance, when $\bH = R^mf^\st_*\bZ$ where $f^\st : X^\st \to B^\st$ is the smooth part of a locally projective morphism $f : X \to B$~\cite[Theorem 3.15]{VoisinII}.

Let  
$\cH \colonec \bH \otimes \cO_{B^\st}$
 and $\cF^\bullet \cH$ be the Hodge filtration on $\cH$. Let 
$$\nabla : \cH \to \cH \otimes \gO^1_{B^\st}$$ 
denote the Gauss-Manin connection. Assume that the local monodromies of $\bH$ around $ B \bss B^\st$ are quasi-unipotent, then $\cH$ has a canonical extension $\bar{\cH}$ due to Deligne together with a filtration $\cF^\bullet \bar{\cH}$ and a morphism 
$$\bar{\nabla} : \bar{\cH} \to \bar{\cH}  \otimes \gO^1_{{B}}(\log{\gS})$$ 
extending $\cF^\bullet \cH$ and $\nabla$ (see \eg~\cite[p.128, p.129]{ZuckerDHB}).
The connection $\bar{\nabla}$ satisfies Griffiths' transversality
$$\bar{\nabla}\(\cF^p \bar{\cH}\) \subset \cF^{p-1} \bar{\cH} \otimes  \gO^1_B(\log{\gS})$$
so that  $\bar{\nabla}$ induces a map 
$$\bar{\nabla}_p : \bar{\cH} / \cF^p \bar{\cH} \to  \(\bar{\cH} / \cF^{p-1} \bar{\cH}\) \otimes  \gO^1_B(\log{\gS}).$$
for each $p$. We define
$$\( \bar{\cH} / \cF^p \bar{\cH}\)_h \colonec \ker (\bar{\nabla}_p )$$
to be the \emph{horizontal part} of $ \bar{\cH} / \cF^p \bar{\cH}$.

The inclusion $\bH \hto \cH$ can be extended to $j_*\bH \hto \bar{\cH}$ over $B$.
Since the local sections of $\bH$ are flat with respect to the Gauss-Manin connection and $\bar{\cH} / \cF^p \bar{\cH}$ is locally free~\cite[p.130]{ZuckerDHB}, the image of the composition $j_*\bH \hto \bar{\cH} \to \bar{\cH} / \cF^p \bar{\cH}$ lies in $\(\bar{\cH} / \cF^p \bar{\cH}\)_h$.

We will only be interested in the case where $i=1$ in Theorem~\ref{thm-ZuckSH}. The following result can be found in the proof of~\cite[Theorem 9.2]{ZuckerHdgL2} as a simple corollary of~\cite[Proposition 9.1]{ZuckerHdgL2}.
\begin{pro}[Zucker]\label{pro-Zuck}
Let $\bH$ be as in Theorem~\ref{thm-ZuckSH} and let  $F^\bullet H$ denote the Hodge filtration on 
$$H \colonec H^1({B} , j_* \bH) \otimes \bC.$$ Then for every integer $p$, the map  $j_*\bH \to \(\bar{\cH} / \cF^p \bar{\cH}\)_h$ induces an isomorphism 
$$H / F^p H \simeq H^1({B}, \(\bar{\cH} / \cF^p \bar{\cH}\)_h).$$ 
\end{pro}

What will be useful for our purpose is the following corollary.

\begin{cor}\label{Cor-Zuck}
Let $\bH$ be as in Theorem~\ref{thm-ZuckSH}. Assume that $\bH$ underlies a VHS of weight $2g-1$  concentrated in bi-degrees $(g,g-1)$ and $(g-1,g)$ and satisfying the Hodge symmetry. Then the image of  
$$H^1({B} , j_* \bH \otimes \bQ) \to H^1({B}, \bar{\cH} / \cF^g \bar{\cH})$$
induced by $ j_* \bH \to  \bar{\cH} / \cF^g \bar{\cH}$ is dense in $H^1({B}, \bar{\cH} / \cF^g \bar{\cH})$.
\end{cor}

\begin{proof}
Since $\bH$ is concentrated in bi-degrees $(g,g-1)$ and $(g-1,g)$, we have $ \cF^{g-1} \bar{\cH} = \bar{\cH}$, so $\bar{\nabla}_g = 0$ and
$$\( \bar{\cH} / \cF^g \bar{\cH}\)_h = \bar{\cH} / \cF^g \bar{\cH}.$$
Thus by Proposition~\ref{pro-Zuck}, we have $ H/F^gH \simeq  H^1({B}, \bar{\cH} / \cF^g \bar{\cH})$.
Since the VHS which overlies $\bH$ satisfies the Hodge symmetry and is concentrated in bi-degrees $(g,g-1)$ and $(g-1,g)$, the Hodge structure $H = H^1({B} , j_* \bH)$ satisfies also Hodge symmetry and is concentrated in bi-degrees $(g+1,g-1)$, $(g,g)$, and $(g-1,g+1)$ by~\cite[Theorem 7.12]{ZuckerHdgL2}. Thus the natural map $H^1({B} , j_* \bH) \otimes \bR \to H/F^gH \simeq  H^1({B}, \bar{\cH} / \cF^g \bar{\cH})$ is surjective. Corollary~\ref{Cor-Zuck} now follows from the density of $H^1({B} , j_* \bH \otimes \bQ) = H^1({B} , j_* \bH) \otimes \bQ$  in $H^1({B} , j_* \bH) \otimes \bR$.
\end{proof}%

\ssec{The canonical extension of $\cJ$}\label{ssec-genjac}
\hfill

 The variations of Hodge structures that we will encounter in the proof of Theorem~\ref{thm-AbFibDefprec} will satisfy the assumption of Corollary~\ref{Cor-Zuck}. In order to simplify the notation, we set $\cE \cnec \cE_{\bH/B} \colonec \cH / \cF^g\cH$, and similarly $\bar{\cE} \cnec \bar{\cE}_{\bH/B} \colonec \bar{\cH} / \cF^g\bar{\cH}$. For these VHSs, recall that the sheaf of sections of the intermediate Jacobian fibration $J$ associated to $\bH$ is defined by $\cJ \cnec \cJ_{\bH/B} \cnec \cE/ \bH$, and the canonical extension of $\cJ$ is defined by the exact sequence~\cite[(2.5)]{ZuckerGenJac}
\begin{equation}\label{SE-genJac}
\begin{tikzcd}[cramped, row sep = 2.5]
0  \arrow[r] & j_*\bH  \arrow[r] &  \bar{\cE} \arrow[r,"\exp"] &  \bar{\cJ} \cnec \bar{\cJ}_{\bH/B} \arrow[r] & 0. 
\end{tikzcd}
\end{equation}

For the VHSs coming from a family of varieties, there is another way to define $\bar{\cJ}$ due to Deligne, El Zein, and Zucker and we recall the construction following~\cite{ElZeinZuck}. Let $f : X \to B$ be a morphism of compact K\"ahler manifolds with $\dim B = 1$. Assume that the fibers of $f$ are normal crossing divisors. Let $j : B^\st \hookrightarrow B$ be a Zariski open subset parameterizing smooth fibers of $f$ and set $ \imath : X^\st \colonec f^{-1}(B^\st) \hto X$. Let $f^\st : X^\st \to B^\st$ denote the restriction of $f$ to  $X^\st$. We assume that the local monodromies of $R^{2g-1}f^\st_*\bZ$ around $\gS = B \bss B^\st$ are quasi-unipotent.

Let $D \colonec f^{-1}(\gS)$. For $\bH \colonec R^{2g-1}f_*^\st \bZ$, we have 
$$\bar{\cH} \simeq R^{2g-1}f_* \gO^\bullet_{X/B}(\log{D})$$
and
$$\cF^p \bar{\cH} \simeq R^{2g-1}f_* F^p\gO^\bullet_{X/B}(\log{D}),$$
where $F^\bullet\gO^\bullet_{X/B}(\log{D})$ is the naïve filtration on $\gO^\bullet_{X/B}(\log{D})$ (\cf\cite[p.130]{ZuckerDHB}). 
Let 
$$\gs_p \gO^\bullet_{X/B}(\log{D}) \colonec \gO^\bullet_{X/B}(\log{D}) /F^p\gO^\bullet_{X/B}(\log{D}).$$
In the bounded derived category of sheaves of abelian groups over $X$, we define the morphism 
$$\Phi^p : R\imath_*\bZ \to  R\imath_*\bC \simeq \gO_X^\bullet (\log{D}) \to  \gO^\bullet_{X/B}(\log{D}) \to \gs_p \gO^\bullet_{X/B}(\log{D})$$
and define the relative Deligne complex to be 
$$\ul{D}_{X/B}(p) \colonec \Cone(\Phi^p)[-1].$$ 
Applying $Rf_*$ to $\ul{D}_{X/B}(p)$ yields the long exact sequence
\begin{equation}\label{LSE-DeligneJac}
\begin{tikzcd}[cramped, row sep = 2.5]
\cdots  \arrow[r] & R^{q-1}f_*\gs_p \gO^\bullet_{X/B}(\log{D})  \arrow[r] & R^qf_*\ul{D}_{X/B}(p) \arrow[r] & R^q(f \circ \imath)_* \bZ \arrow[r] & \cdots 
\end{tikzcd}
\end{equation}
The canonical extension of ${\cJ}$ is defined to be
$$\bar{\cJ} \colonec \coker \(R^{2g-1}(f \circ \imath)_* \bZ \to R^{2g-1}f_*\gs_g \gO^\bullet_{X/B}(\log{D}) \).$$

Note that since the spectral sequence 
$$E_1^{p,q} = R^q f_* \gO^p_{X/B}(\log{D}) \Rightarrow R^{p+q}f_* \gO^\bullet_{X/B}(\log{D})$$
degenerates at $E_1$ (see \cite[p.130]{ZuckerDHB}), we have  $R^{2g-1}f_*\gs_g \gO^\bullet_{X/B}(\log{D}) \simeq \bar{\cE}$. So for $p = g$, breaking up the long exact sequence~\eqref{LSE-DeligneJac} at $R^{2g-1}f_*\gs_g \gO^\bullet_{X/B}(\log{D})$ yields
\begin{equation}
\begin{tikzcd}[cramped, row sep = 2.5]
0  \arrow[r] & \Ima\(R^{2g-1}(f \circ \imath)_* \bZ \xto{\psi} \bar{\cE}\)  \arrow[r] &  \bar{\cE} \arrow[r] &  \bar{\cJ}  \arrow[r] & 0. 
\end{tikzcd}
\end{equation}
In order to compare the above construction of $\bar{\cJ}$ with the one defined by~\eqref{SE-genJac}, it suffices to prove the following isomorphism. 

\begin{lem}\label{lem-iso2contr}
We have the isomorphism
\begin{equation}\label{iom-defcanextjac}
j_*\bH \simeq \Ima\(\psi : R^{2g-1}(f \circ \imath)_* \bZ \to \bar{\cE}\).
\end{equation}
\end{lem}

Since a proof of Lemma~\ref{lem-iso2contr} is hard to find in the literature, we provide one below.

\begin{proof}
As the restriction of $\psi$ to $B^\st$ is the inclusion $\bH \hto \cE$, it suffices to prove~\eqref{iom-defcanextjac} around a neighborhood $p \in B \bss B^\st$. Let $\Delta \subset B$ a sufficiently small disc centered at $p$. Let $\gD^\st \colonec \Delta - \{p\}$ and $X^\st_\Delta \colonec f^{-1}(\Delta^\st)$. Since the natural morphism  $\phi : R^{2g-1}(f \circ \imath)_* \bZ \to j_*\bH$ is surjective (\emph{cf.} the proof of~\cite[Proposition 7]{ElZeinZuck}),  it suffices to show that 
\begin{equation}\label{ker=}
\ker(\phi) = \ker(\psi)
\end{equation} 
over  $\Delta$. 
Let $\ga \in H^{2g-1}(X_\gD^\st, \bZ)$. On the one hand,
$$\ga \in \ker\(\phi(\Delta) : H^{2g-1}(X_\gD^\st, \bZ) \to H^0(\gD^\st, R^{2g-1}f_*\bZ) \)$$
if and only if its restriction to a general fiber $X_s$ of $X^\st_\Delta \to \Delta^\st$, which is an element  in 
$$H^{2g-1}(X_s,\bZ) \subset H^{2g-1}(X_s, \bC)/F^gH^{2g-1}(X_s, \bC),$$
is zero. On the other hand, since $\bar{\cE}$ is locally free~\cite[p. 130]{ElZeinZuck}, 
$$\ga \in \ker\(\psi(\Delta) : H^{2g-1}(X_\gD^\st, \bZ) \to H^0(\Delta, \bar{\cE} )  \) $$
if and only if the restriction of $\psi(\Delta)(\ga)$ to a general fiber $\bar{\cE}_s \simeq H^{2g-1}(X_s, \bC)/F^gH^{2g-1}(X_s, \bC)$ of  $\bar{\cE}$ is zero. Since the two restrictions coincide, we have~\eqref{ker=}.
\end{proof}

We can also break up~\eqref{LSE-DeligneJac} at $R^{2g}f_*\ul{D}_{X/B}(g)$, which yields the short exact sequence
\begin{equation}\label{SE-DeligneJac}
\begin{tikzcd}[cramped, row sep = 2.5]
0  \arrow[r] & \bar{\cJ}  \arrow[r] & R^{2g}f_*\ul{D}_{X/B}(g) \arrow[r] & H^{g,g}(X/B) \arrow[r] & 0 
\end{tikzcd}
\end{equation}
where
$$H^{g,g}(X/B) \colonec \ker \( R^{2g}(f \circ \imath)_* \bZ \to R^{2g}f_*\gs_g \gO^\bullet_{X/B}(\log{D}) \).$$

\begin{lem}\label{lem-comm}
The diagram
\begin{equation}\label{diag-chernleray}
\begin{tikzcd}[cramped]
 H^0\(B, R^{2g}({f} \circ \imath)_*\bZ\)   \arrow[rr,"d_2"] & & H^2\(B, R^{2g-1}(f \circ \imath)_*\bZ\) \arrow[d, two heads]  \\
H^0(B, H^{g,g}({X}/B)) \arrow[u]   \arrow[r, "\delta"] & H^1(B, \bar{\cJ})   \arrow[r,"c"] & H^2(B,j_*\bH)  
\end{tikzcd}
\end{equation}
is commutative, where the first row is the $d_2$ map in the second page of the Leray spectral sequence
$$E_2^{p,q} = H^p\(B, R^{q}({f} \circ \imath)_*\bZ\)  \Rightarrow H^{p+q}(X^\st,\bZ)$$
and the morphisms in the second row are the connecting morphisms induced by~\eqref{SE-DeligneJac} and~\eqref{SE-genJac} respectively.
\end{lem}

\begin{proof}
It suffices to apply Lemma~\ref{lem-appenComm} (\cf Appendix)
to the short exact sequence of bounded complexes representing the distinguished triangle
\begin{equation}
\begin{tikzcd}[cramped, row sep = 2.5]
Rf_*\ul{D}_{X/B}(g) \arrow[r] & Rf_*R\imath_*\bZ \arrow[r] & Rf_*\gs_g \gO^\bullet_{X/B}(\log{D}) \arrow[r] & Rf_*\ul{D}_{X/B}(g)[1] 
\end{tikzcd}
\end{equation}
and note that  $j_*\bH \simeq \Ima\(R^{2g-1}(f \circ \imath)_* \bZ \to \bar{\cE}\)$ by Lemma~\ref{lem-iso2contr}.
\end{proof}%

When $g$ is the relative dimension of the fibration $f: X \to B$, we have the following result. 
\begin{lem}\label{lem-isoFibg}
Let $g = \dim X - \dim B$. We have
$$H^{g,g}(X/B) = R^{2g}(f \circ \imath)_* \bZ.$$
In particular, the morphism $H^0\(B, H^{g,g}(X/B)\) \to H^0\(B, R^{2g}({f} \circ \imath)_*\bZ\)$ in~\eqref{diag-chernleray} is the identity.
\end{lem}

\begin{proof}
By definition of $H^{g,g}({X}/B)$, it suffices to show that the sheaf $R^{2g}f_*\gs_g \gO^\bullet_{X/B}(\log{D})$ is zero.
Let $\bV \colonec R^{2g}f_*^\st \bZ$. Since $f:X \to B$ is of relative dimension $g$, $\cV$ is a pure VHS concentrated in bi-degree $(g,g)$.  The sheaf $\bar{\cV} \colonec R^{2g}f_* \gO^\bullet_{X/B}(\log{D})$ is the canonical extension of $\cV \colonec\bV \otimes \cO_{B^\st}$, so $R^{2g}f_*\gs_g \gO^\bullet_{X/B}(\log{D}) = \bar{\cV}/F^g\bar{\cV}$ is locally free~\cite[p.130]{ElZeinZuck}. It follow from ${\cV}/F^g\cV = 0$ that $\bar{\cV}/F^g\bar{\cV} = 0$.
\end{proof}%

\ssec{Bimeromorphic $J$-torsors and their tautological families}\label{ssec-bimJtors}
\hfill

Bimeromorphic $J$-torsors are defined as follows.

\begin{Def}\label{def-bimJtors}
Let $\phi : J \to B^\st$ be a Jacobian fibration over a smooth quasi-projective curve $B^\st$ and $B$ the smooth compactification of $B^\st$. 
Assume that $\phi$ is projective. 
A \emph{bimeromorphic $J$-torsor} is a 
proper morphism $f : X \to B$ satisfying the following properties:
\begin{enumerate}[label = \roman{enumi})]
\item\label{def-bimJtors-openJtors} The restriction $f^\st : X^\st \to B^\st$ of $f$ to $X^\st  \cnec f^{-1}(B^\st)$ is a $J$-torsor.
\item The total space $X$ is locally bimeromorphically K\"ahler over $B$. Namely, for every $b \in B$, there exists a neighborhood $U \subset B$ of $b$ such that $f^{-1}(U)$ is bimeromorphic to a K\"ahler manifold.
\item The fibration $f$ has local sections at every point of $B$. Namely, for every $b \in B$, there exists a neighborhood $U \subset B$ of $b$ such that $f$ has a local section $\gs_U : U \to f^{-1}(U)$ over $U$.
\end{enumerate}
\end{Def}%

The following lemma allows us to apply results obtained in \S\ref{ssec-thmZucker} and \S\ref{ssec-genjac} to study bimeromorphic $J$-torsors. An (integral) local system $\bH$ over a Zariski open $B^\st$ of a smooth curve $B$ is called \emph{locally geometric} if for every $b \in B$, there exists a neighborhood $U \subset B$ of $b$ and a projective morphism $f_U : X_U \to U$ such that $\bH_{|U \cap B^\st} \simeq \(R^i{f_U}_*\bZ\)_{|U \cap B^\st}$ for some $i \in \bZ$.

\begin{lem}\label{lem-locgeom}
	 If $f : X \to B$ is a bimeromorphic $J$-torsor, then the VHS which overlies $\bH \cnec R^{2g-1}f_*^\st\bZ$ is locally geometric where $g = \dim X - \dim B$. In particular, the local monodromies of $\bH$ around $B \bss B^\st$ are quasi-unipotent.
\end{lem}

\begin{proof}
On the one hand, since $f$ is locally bimeromorphically K\"ahler and fibers of $f$ are algebraic, by~\cite[Corollaire du Théorème 2]{Campana-redalg}  $f$ is locally Moishezon. On the other hand, there exists by assumption some modification $\ti{f} : \ti{X} \to B$ of $f$ along singular fibers such that for every $b \in B$, there exists a neighborhood $U \subset B$ of $b$ such that $\ti{f}^{-1}(U)$ is K\"ahler. So $\ti{f}$ is locally projective by~\cite[Theorem 10.1]{CampanaPeternell2-form} and since $R^{2g-1}\ti{f}_*\bZ_{|B^\st} = R^{2g-1}f_*\bZ_{|B^\st} = \bH$, it follows that  $\bH$ is locally  geometric.
\end{proof}

Let $f : X \to B$ be a bimeromorphic $J$-torsor. When $B \bss B^\st \ne \emptyset$, we let $\{p_1,\ldots,p_m\} \cnec B \bss B^\st$. By assumption, there exists a good open cover $\{U_i\}_{i=1,\ldots,n}$ of $B$ such that $p_i \in U_j$ if and only if $i=j$ and that $f_i : X_i \colonec f^{-1}(U_i) \to U_i$ has a section  $\gs_i : U_i\to X_i$ for all $i$. For each $i > m$ (resp. $i\le m$), let $U^\st_i = U_i$ (resp. $U^\st_i = U_i - \{p_i\}$). By Definition~\ref{def-bimJtors} i) and iii), 
for each $i$ there exists a biholomorphic map
\begin{equation}\label{eqn-etai}
\eta_i : X^\st_i = f^{-1}(U^\st_i) \to \phi^{-1}(U^\st_i) = J_i
\end{equation}
 over $U^\st_i$ sending $\gs_i (U^\st_i)$ to the zero-section of $J_i \to U^\st_i$,
 such that over $U_{ij} \cnec U_i \cap U_j = U^\st_i \cap U^\st_j$ (for $i \ne j$),
 $$\eta_j \circ \eta_i^{-1} = \tr(\eta_{ij}) : J_{ij} \to J_{ij} \cnec p^{-1}(U_{ij})$$
 for some 1-cocycle $\{\eta_{ij}\}$ with coefficients in $\cJ$ 
 (hence in $\bar{\cJ}$) defining the $J$-torsor structure of $f^\st : X^\st \to B^\st$.
  Whenever $i \ne j$, we have $U_i \cap U_j = U^\st_i \cap U^\st_j$, so we can glue the fibrations $X_i \to U_i$ along $X_{ij} \cnec f^{-1}(U_{ij})$ using the transition maps 
$$\eta_{j}^{-1} \circ \eta_{i} : X_{ij} \to X_{ij}.$$
We call the thus obtained fibration $p : \bar{J} \to B$ the \emph{Jacobian fibration associated to $f :X\to B$}, which is a compactification of  $J \to B^\st$. By construction, the zero-section of $J \to B^\st$ extends to a section of $\bar{J} \to B$.
Note also that by construction, $\eta_i$ in~\eqref{eqn-etai} extends to a biholomorphic map
\begin{equation}\label{eqn-etai'}
\eta_i : X_i \eto p^{-1}(U_i) =: \bar{J}_i
\end{equation}
over $U_i$.

\begin{lem}\label{lem-Jbarhol}
The bimeromorphic class of the Jacobian fibration $p : \bar{J} \to B$ constructed above is independent of the bimeromorphic $J$-torsor $f$. Also, the sheaf $\bar{\cJ}$ is contained in the sheaf of germs of holomorphic sections of $p : \bar{J} \to B$.
\end{lem}

\begin{rem}
	 The sheaf of germs of holomorphic sections of $\bar{J} \to B$ is a sheaf of abelian groups. Indeed, since $f$ is a fibration over a smooth curve, 
	 every local meromorphic section is holomorphic 
	 (because every meromorphic map from a smooth curve is holomorphic). 
	 For every open subset $U \subset B$, meromorphic sections of $p : \bar{J} \to B$ over $U$ form an abelian group coming from the group variety structure of $J \to B^\st$.
	
\end{rem}

\begin{proof}
By Lemma~\ref{lem-locgeom}, the VHS which overlies $\bH$ is locally geometric and is therefore admissible~\cite[Theorem 14.51]{PetersSteenbMHS}. Accordingly, the Jacobian fibration $J \to B^\st$ admits a Saito compactification $\bar{J}_S \to B$~\cite[Theorem 0.8]{MSaitoAdmnorm} and there is a bimeromorphic map $\nu : \bar{J} \dto \bar{J}_S$ over $B$~\cite[Example 2.10]{MSaitoAdmnorm}. This proves the first statement of Lemma~\ref{lem-Jbarhol}.

As $\bar{J}_S \to B$ is a compactification of Zucker's extension $\bar{J}_Z \to B$~\cite[p.237--238]{MSaitoAdmnorm} constructed in~\cite[Section 2]{ZuckerGenJac} and since $\bar{\cJ}$ is contained in the sheaf of germs of holomorphic sections of  $\bar{J}_Z \to B$~\cite[Section 2]{ZuckerGenJac},  we conclude that $\bar{\cJ}$ is identified (\emph{via} $\nu$) with a subsheaf of sheaf of germs of holomorphic sections of  $p : \bar{J} \to B$.
\end{proof}

Let $\cE(B,\gD,\bH)$ be the set of bimeromorphic classes (over $B$) of bimeromorphic $J$-torsors. Assume that $\cE(B,\gD,\bH) \ne \emptyset$, then we can construct a map 
$$\Phi : H^1(B,\bar{\cJ}) \to \cE(B,\gD,\bH)$$
as follows. First we fix a Jacobian fibration $p : \bar{J} \to B$ associated to some element in $\cE(B,\gD,\bH)$ and a good open cover $\{U_i\}$ of $B$ such that $U_{ij} \cnec U_i \cap U_j \subset B^\st$ for every $i$ and $j$. Let $\eta \in H^1(B,\bar{\cJ})$, represented by a 1-cocycle $\{\eta_{ij}\}$ defined over $U_{ij}$. 
Since $U_{ij} \subset B^\st$, the sections $\eta_{ij}$ define a 1-cocycle of translations
$$\tr(\eta_{ij}) : p^{-1}(U_{ij}) \eto p^{-1}(U_{ij})$$
over $U_{ij}$ and can be used to glue the $ p^{-1}(U_i) \to U_i$ and form a bimeromorphic $J$-torsor $f : X \to B$. 
We call $f$ the \emph{bimeromorphic $J$-torsor twisted by $\{\eta_{ij}\}$}.
Suppose that $\{\eta'_{ij}\}$ is another 1-cocycle representing $\eta$ and let $f' : X' \to B$ be the bimeromorphic $J$-torsor twisted by $\{\eta'_{ij}\}$. Let $\gt_i \in \bar{\cJ}(U_i)$ be local sections such that ${\gt_i}_{|U_{ij}} - {\gt_j}_{|U_{ij}} = \eta_{ij} - \eta'_{ij}$. As $\gt_i$ is a holomorphic section of $p: \bar{J} \to B$ (Lemma~\ref{lem-Jbarhol}), the translation map $\tr(\gt_i)$, which is biholomorphic over $U_i \cap B^\st$, extends to a bimeromorphic map $\tr(\gt_i) : p^{-1}(U_i) \dto p^{-1}(U_i)$ over $U_i$~\cite[Proposition 1.6]{NakayamaWeierCnUniv}. 
These bimeromorphic maps glue together and
define a bimeromorphic map $X \dto X'$ over $B$, which shows that $f$ and $f'$ have the same image in $\cE(B,\gD,\bH)$. Thus $\Phi$ is well-defined. Finally, it is easy to see that  the construction of $\Phi$ does not depend on the choice of the Jacobian fibrations $\bar{J} \to B$, because these fibrations have the same bimeromorphic class (Lemma~\ref{lem-Jbarhol}).
We also note that if $\eta, \eta' \in H^1(B,\bar{\cJ})$ satisfy $\Phi(\eta) = \Phi(\eta')$,
then $\Phi(\eta + \gt) = \Phi(\eta' + \gt)$ for any $\gt \in H^1(B,\bar{\cJ})$.

\begin{lem}\label{lem-surjPhi}
	The map $\Phi : H^1(B,\bar{\cJ}) \to \cE(B,\gD,\bH)$ is surjective.	
\end{lem}

\begin{proof}
	Let $f : X \to B$ be a bimeromorphic $J$-torsor and let
	$p: \bar{J} \to B$ be a Jacobian fibration associated to $f$.
	For every $i$,
	let $\eta_i : X_i \eto \bar{J}_i$
	be the biholomorphic map~\eqref{eqn-etai'}.
	Then $\eta_j \circ \eta_i^{-1} = \tr(\eta_{ij}) : p^{-1}(U_{ij}) \to  p^{-1}(U_{ij})$
	over $U_{ij}$ for some 1-cycle $\eta = \{ \eta_{ij}\}$ with coefficients in $\bar{\cJ}$.
	By construction, we have $\Phi(\eta) = [f]$.
\end{proof}

Given an integer $m \in \bZ$ and a (smooth) $J$-torsor,
we can define its multiplication-by-$m$, which is another $J$-torsor
(see \eg~\cite[p.482]{ClaudonToridefequiv}).
The construction of multiplication-by-$m$ can be extended to
bimeromorphic $J$-torsors, which we explain now.

Let $f : X \to B$ be a bimeromorphic $J$-torsor.
As before, let $\{U_i\}_{i \in I}$ be a good open cover of $B$ 
such that $U_{ij} \cnec U_i \cap U_j \subset B^\st$ for every $i$ and $j$.
We start by choosing a
1-cycle of local multi-sections $\gs_i$ over each $U_i$ of degree $m$,
namely a $1$-cycle
$$\gs_i = \sum_{\ell \in J_i} c_{i,\ell} \cdot Z_{i,\ell} \in Z_1(X_i)$$
in each $X_i \cnec f^{-1}(U_i)$, such that each $Z_{i,\ell}$ is a 
multi-section over $U_i$ and $c_{i,\ell} \in \bZ$ 
satisfying 
$$\deg(\gs_i) \cnec \sum_{\ell \in J_i} c_{i,\ell} \cdot d_{i,\ell} = m,$$
where $d_{i,\ell}$ is the degree of the finite map $Z_{i,\ell} \xto{f} U_i$.
Such a choice exists because $f : X\to B$ has local sections everywhere over $B$ by assumption.
 
Let $U_{i}^\circ \subset U_{i}$ be a nonempty Zariski open subset 
such that all the multi-sections $Z_{i,\ell}$ are \'etale over $U^\circ_{i}$.
For every $t \in U_{i}^\circ$, define 
$$\left\{x_{i,\ell,1}(t),\ldots x_{i,\ell,d_{i,\ell}}(t) \right\} \cnec Z_{i,\ell} \cap X_t$$ 
where $X_t \cnec f^{-1}(t)$. 
Given $i , j \in I$, since 
$$\deg(\gs_i) - \deg(\gs_j) = \sum_{\ell \in J_i} c_{i,\ell} \cdot d_{i,\ell} -\sum_{\ell \in J_j} c_{j,\ell} \cdot d_{j,\ell} = m - m  = 0,$$
we can define the
Albanese image
$$\gs_{ij}(t) \cnec \gs_{i}(t) -  \gs_{j}(t) \cnec 
\( \sum_{\ell \in J_i} c_{i,\ell} \cdot \sum_{k = 1}^{d_{i,\ell}} x_{i,\ell,k}(t) \)
- \(\sum_{\ell \in J_j} c_{j,\ell} \cdot \sum_{k = 1}^{d_{j,\ell}} x_{j,\ell,k}(t)\)
\in \Alb(X_t) =  J_t \cnec p^{-1}(t)$$
for every $t\in  U^\circ_{ij} \cnec U_{i}^\circ \cap U_j^\circ$.
Each $\gs_{ij}$ defines a local section of $p : \bar{J} \to B$ over $U^\circ_{ij}$,
which extends to a local section over $U_{ij}$.
By construction, the 1-cochain $\gs \cnec \{\gs_{ij}\}$ with coefficients in $\bar{\cJ}$ 
is a 1-cocycle.
Let $f^\gs : X^\gs \to B$ be the bimeromorphic $J$-torsor 
twisted by $\gs$.

\begin{lem-def}\label{lemdef-multm}
The bimeromorphic class $\Phi(\gs) = [f^\gs] \in \cE(B,\gD,\bH)$
is independent of $\gs$, as long as
each local 1-cycle of multi-sections $\gs_i$ 
defining $\gs$ satisfies $\deg(\gs_i) = m$.
Moreover, if $\eta \in H^1(B,\bar{\cJ})$ satisfies
$\Phi(\eta) = [f]$, then $\Phi(m\cdot \eta) = [f^\gs]$.

A bimeromorphic $J$-torsor $f_m : X_m \to B$ 
representing $[f^\gs] \in \cE(B,\gD,\bH)$ is called a \emph{multiplication-by-$m$} of $f : X \to B$,
which coincides with the usual definition 
(see \eg~\cite[p.482]{ClaudonToridefequiv}) when $f$ is smooth.
\end{lem-def}

\begin{proof}
	Let
	$$\gs'_i = \sum_{\ell \in J'_i} c'_{i,\ell} \cdot Z'_{i,\ell} \in Z_1(X_i)$$
	be another collection of 1-cycles of multi-sections 
	over $U_i$ such that $\deg(\gs_i') = m$ for each $i$.
	For every $i$ and $t \in U_i^\circ$, consider the Albanese image
	$$\gt_i(t) \cnec  \gs_{i}(t) -  \gs'_{i}(t) \cnec 
	\( \sum_{\ell \in J_i} c_{i,\ell} \cdot \sum_{k = 1}^{d_{i,\ell}} x_{i,\ell,k}(t) \)
	- \(\sum_{\ell \in J'_i} c'_{i,\ell} \cdot \sum_{k = 1}^{d'_{i,\ell}} x'_{i,\ell,k}(t)\)
	\in \Alb(X_t) =  J_t $$ 
	where $d'_{i,\ell}$ and $x'_{i,\ell,k}(t)$ are defined for $\gs'$
	the same way we define $d_{i,\ell}$ and $x_{i,\ell,k}(t)$ for $\gs$.
	
	Each $\gt_i$ defines a local section of 
	$p : \bar{J} \to B$ over $U^\circ_{i}$,
	which extends to a local section over $U_{i}$,
	so the translation by $\gt_i$ defines a 
	bimeromorphic map $\tr(\gt_i) : \bar{J}_i \dto \bar{J}_i \cnec p^{-1}(U_i)$
	by~\cite[Proposition 1.6]{NakayamaWeierCnUniv}.
	By construction, we have the commutative diagram of isomorphisms
	\begin{equation}
	\begin{tikzcd}
	\bar{J}_{ij} \cnec p^{-1}(U_{ij}) \ar[d, "\tr(\gt_i)"'] \ar[r, "\tr(\gs_{ij})"] & \bar{J}_{ij} \ar[d, "\tr(\gt_j)"] \\
	\bar{J}_{ij}  \ar[r, "\tr(\gs'_{ij})"] & \bar{J}_{ij}, 
	\end{tikzcd}
	\end{equation}
	so we can glue the bimeromorphic maps $\tr(\gt_i) : \bar{J}_i \dto \bar{J}_i$ along $\bar{J}_{ij}$
	through $\tr(\gs_{ij})$ and $\tr(\gs'_{ij})$, and form a bimeromorphic map
	$X^\gs \dto X^{\gs'}$ over $B$.
	
	Now we prove the second statement. By assumption,
	there exist a 1-cocycle $\{\eta_{ij}\}$ representing 
	$\eta \in H^1(B,\bar{\cJ})$
	and bimeromorphic maps 
	$\eta_i :  X_i \dto p^{-1}(U_i) =: \bar{J}_i$ 
	over $U_i$
	such that  $\eta_i$ is an isomorphism over $U_i \cap B^\st$ and
	$\eta_i \circ \eta_j^{-1} = \tr(\eta_{ij})$ over $U_{ij}$.
	Let $Z_i \subset X_i$ be the image of the 0-section of 
	$\bar{J}_i \to U_i$ under $\eta_i$. 
	By choosing
	$\gs_i \cnec m \cdot Z_i$ in the definition of $\gs$,
	we have $\gs = m \cdot \eta$, 
	which proves that $\Phi(m \cdot \eta) = [f^\gs]$.
\end{proof}

%

Let $m \in \bZ_{>0}$ and let $f_m : X_m \to B$ be a multiplication-by-$m$ of $f : X \to B$. 
There exists a finite étale morphism $X^\st \to X_m^\st$ called \emph{multiplication-by-$m$}~\cite[p.482]{ClaudonToridefequiv} defined on the smooth part of $f$ and by~\cite[Proposition 1.6]{NakayamaWeierCnUniv}, this map has a meromorphic extension $\bm : X \dto X_m$ over $B$. As a consequence, we obtain the following result.

\begin{lem}\label{lem-torsmulsc}
	If $\eta \in H^1(B,\bar{\cJ})$ is a torsion element, then the bimeromorphic $J$-torsors represented by $\Phi(\eta)$ have  multi-sections.
\end{lem}

\begin{proof}
	
	Let $f:X \to B$ be a bimeromorphic $J$-torsor represented by $\Phi(\eta)$. 
	If $m \eta = 0$ for some $m \in \bZ_{>0}$, 
	then since $\Phi(0) = [p : \bar{J} \to B]$,
	 we have a multiplication-by-$m$ map $\bm : X \dto \bar{J}$ over $B$. As $\bm$ is generically finite,  
	the pre-image of a global section of $\bar{J} \to B$ under $\bm$ is a  multi-section of $X \to B$. 
\end{proof}

The following lemma 
could be considered as a version of~\cite[Remark 2.3]{ClaudonToridefequiv}
for bimeromorphic $J$-torsors.

\begin{lem}\label{lem-etadef}
Let $f : X \to B$ be a bimeromorphic $J$-torsor. 
Assume that $X$ is smooth and
the singular fibers of $f$ are normal crossing divisors.
Let $\imath : X^\st  \hto X$ be the inclusion.
Let $\gb \in H^0(B, R^{2g}f_* \bZ)$
and let $n \in \bZ$ denote the image of the restriction $H^0(B, R^{2g}f_* \bZ) 
\to H^0(B^\st, R^{2g}f^\st _* \bZ) \simeq \bZ$. 
Then there exists $d \in \bZ_{>0}$ such that
if $\eta \in H^1(B,\bar{\cJ})$ denotes the image of $d \cdot \gb$ under the composition
\begin{equation}\label{mor-etadef}
\begin{tikzcd}[cramped, row sep = 4]
H^0(B, R^{2g}f_* \bZ) \ar[r] & H^0(B, R^{2g}(f \circ \imath)_* \bZ) \simeq H^0(B, H^{g,g}(X/B)) \arrow[r, "\gd"] & H^1(B,\bar{\cJ})   
\end{tikzcd}
\end{equation}
induced by Lemma~\ref{lem-isoFibg} and the short exact sequence~\eqref{SE-DeligneJac}, then $\Phi(\eta) = [f_{dn}] \in \cE(B,\gD,\bH)$.
\end{lem}

%
%

\begin{proof}

First we prove that there exists $d \in \bZ_{>0}$ such that locally, 
$d \cdot \gb$ is some linear combination of classes of multi-sections. 

\begin{lem}\label{lem-Zjcj}
	There exists $d \in \bZ_{>0}$ such that for every $t \in B$, 
	there exist a neighborhood $t \in U \subset B$ of $t$ 
	and a finite number of (holomorphic) multi-sections 
	$Z_j \subset X_U \cnec f^{-1}(U)$ of $X_U \xto{f} U$ 
	together with some integers $c_j$
	such that 
	$$\sum_{j} c_j [Z_j] = d\cdot \gb_{|U} \in H^0(U, R^{2g}f_* \bZ).$$ 
\end{lem}

\begin{proof}
	Let $\gS \subset B$ be the finite subset parameterizing singular fibers of $f$.
	Given $t \in \gS$, 
	let $t \in U \subset B$ be a contractible neighborhood of $t$
	such that $f$ is smooth over $U \bss \{t\}$.
	Then
	$$H^0(U, R^{2g}f_* \bZ) = H^{2g}(f^{-1}(U), \bZ) \simeq H^{2g}(f^{-1}(t), \bZ) 
	\simeq \bigoplus_{F_j \in J} \bZ $$
	where  $J$ is the set of irreducible components of 
	$X_t \cnec f^{-1}(t)$. 
	Through this isomorphism, 
	$\gb_{|U} \in H^0(U, R^{2g}f_* \bZ)$ 
	is sent to the intersection numbers $m_j \cnec \gb_{|U} \cdot [F_j]$ with $F_j$ running through $J$.
	Since $f$ is locally projective by assumption,
	up to shrinking $U$ we can also assume that $f_U \cnec f_{|X_U} $ is projective.
	So for every irreducible component $F_j$, 
	there exists a multi-section $Z_j \subset X_U$ of $f_U : X_U \to U$
	which only intersects with $F_j$ (at $d_j > 0$ points, counted with multiplicity)
	and disjoint with other components $F_\ell$ of $X_t$.
	Let $d_t \cnec \lcm\{d_j \mid F_j \in J\}$ and let $d \cnec \lcm\{d_t \mid t\in \gS\}$. 
	Then 
	$$d \cdot \gb_{|U} = \sum_{j \in J} c_j [Z_j] 
	\in H^0(U, R^{2g}f_* \bZ)$$ 
	with $c_j = m_jd/d_j$, which are integers.
	
	Finally if $t \in B \bss \gS$, then $f$ is smooth over a 
	contractible neighborhood $U \subset B$ of $t$ and we have 
	$H^0(U, R^{2g}f_* \bZ) \simeq H^{2g}(f^{-1}(t), \bZ)$. 
	Up to shrinking $U$, we can assume that $f_U : X_U \to U$ has a section $Z \subset X_U$,
	and thus $dn \cdot [Z] = d \cdot \gb_{|U}$.
\end{proof}
	
	Back to the proof of Lemma~\ref{lem-etadef}, 
	as before let $\{U_i\}$ be a good open cover of 
	$B$ such that $U_{ij} \cnec U_i \cap U_j \subset B^\st$.  
	By Lemma~\ref{lem-Zjcj}, we can write
	$$ d\cdot \gb_{|U_i} = \sum_{j} c_{i,j} [Z_{i,j}]\in H^0(U_i, R^{2g}f_* \bZ)$$
	for some multi-sections $Z_{i,j}$ over $U_i$ 
	and some $c_{i,j} \in \bZ$.
	
	By definition, a 1-cocycle $\{\eta_{ij}\}$ with respect to the open cover $\{U_i\}$ 
	which represents $\eta$ can be constructed as follows. 
	Let $\gb' \in H^0(B, R^{2g}(f \circ \imath)_* \bZ)$ 
	be the image of $\gb \in H^0(B, R^{2g}f_* \bZ)$ under~\eqref{mor-etadef}.
	Then $\gs_i \cnec \sum_j c_{i,j} [Z_{i,j}] \in R^{2g}f_*\ul{D}_{X/B}(g)(U_i)$ 
	is a lift of $d \cdot \gb'_{|U_i}$ by virtue of~\eqref{SE-DeligneJac} and Lemma~\ref{lem-isoFibg}. 
	We define $\{\eta_{ij}\}$ to be the 1-cocycle obtained by taking the \v{C}ech differential of $\{\gs_i\}$, which, by construction, represents $\eta \in H^1(B,\bar{\cJ})$.
	Hence $\Phi(\eta) = [f_{dn}] \in \cE(B,\gD,\bH)$ by Lemma-Definition~\ref{lemdef-multm}
	together with the construction of $\{\eta_{ij}\}$.
\end{proof}

The following lemma is the analogue of~\cite[Proposition 2.5]{ClaudonToridefequiv}
for bimeromorphic $J$-torsors.


\begin{pro}\label{pro-relevtor}
Let $f : X \to B$ be a bimeromorphic $J$-torsor such that $X$ is K\"ahler.
Let $c : H^1(B,\bar{\cJ}) \to H^2(B,j_*\bH)$ be the morphism induced by~\eqref{SE-genJac}. 
There exist $\eta \in H^1(B,\bar{\cJ})$ and $m \in \bZ_{>0}$ 
such that $\Phi(\eta) = [f_m]$ and $c(\eta) = 0 \in H^2(B,j_*\bH)$.
\end{pro}

\begin{proof}
	Up to replacing $X$ by a strong K\"ahler desingularization of it, 
	we can assume that $X$ is a compact K\"ahler manifold and
	the singular fibers of $f$ are normal crossing divisors. 
	The Gysin morphism $f_* : H^{2g}(X,\bQ) \to H^0(B,\bQ)$ has the factorization
\begin{equation}\label{diag-defeta}
\begin{tikzcd}[cramped]
f_* : H^{2g}(X,\bQ)  \arrow[r] & H^0(B, R^{2g}f_* \bQ) \ar[r] & H^0(B,\bQ)
\end{tikzcd}
\end{equation} 
through the restriction to the space of global sections of higher direct image. 
Since $X$ is a compact K\"ahler manifold, 
the composition~\eqref{diag-defeta} is surjective,
so there exists $\ga \in H^{2g}(X,\bZ)$ such that $f_*\ga = n \in \bZ_{>0} \subset H^0(B,\bZ)$.

Let $\imath : X^\st  \hto X$ be the inclusion. 
By Lemmas~\ref{lem-comm} and~\ref{lem-isoFibg} 
we have the commutative diagram
\begin{equation}\label{diag-simp}
\begin{tikzcd}[cramped]
H^{2g}(X,\bZ) \arrow[r] & H^0(B, R^{2g}f_* \bZ)  \arrow[r,"d_2"]  \arrow[d]  & H^2(B, R^{2g-1}f_* \bZ) \arrow[d]  \\
& H^0\(B, R^{2g}(f \circ \imath)_* \bZ\)  \arrow[r, "d_2"]  \arrow[d]  & H^2\(B, R^{2g-1}(f \circ \imath)_* \bZ\) \arrow[d]  \\
&  H^1(B, \bar{\cJ})   \arrow[r,"c"] & H^2(B,j_*\bH).  
\end{tikzcd}
\end{equation}
Let $\eta' \in H^1(B,\bar{\cJ})$ be the image of 
$\ga \in H^{2g}(X,\bZ)$ following the arrows in~\eqref{diag-simp}.
By Lemma~\ref{lem-etadef},
there exists $d \in \bZ_{>0}$ such that if $\eta \cnec d \cdot \eta'$ and $m \cnec dn \in \bZ_{>0}$,
then $\Phi(\eta) = \Phi(f_m)$.

The composition of the arrows of 
the first row of~\eqref{diag-simp} vanishes. 
Indeed, let $\{E_i^{p,q}\}$ denote the Leray spectral sequence associated to 
$f : X \to B$.
As every Leray differential which goes to $E_i^{0,2g}$ is zero whenever $i \ge 2$,
we have $E_\infty^{0,2g} \subset E_3^{0,2g} = \ker(d_2) \subset E_2^{0,2g}$.
Since the first arrow $H^{2g}(X,\bZ) \to H^0(B, R^{2g}f_* \bZ)$
is the composition
$H^{2g}(X,\bZ) \to  E_\infty^{0,2g} \hto E_2^{0,2g}$,
the map vanishes when further composing with $d_2$.
Finally, as $\eta \in H^1(B, \bar{\cJ})$ is the image of some element of $H^{2g}(X,\bZ)$, 
it follows that $c(\eta) = 0$. 
\end{proof}

We will also need the following lemma in the proof of Theorem~\ref{thm-AbFibDefprec}.

\begin{lem}\label{lem-relevtorG}
	Let $f : X \to B$ be a bimeromorphic $J$-torsor and assume that $f$ is $G$-equivariant 
	with respect to the action of some finite group $G$. 
	Let $\eta \in H^1(B,\bar{\cJ})$ be any lift of $f$ by $\Phi$. Then 
	$$\Phi(|G| \cdot \eta) = \Phi\(\sum_{g \in G} g^*\eta\).$$
\end{lem}

\begin{proof}

 As before, let $\{U_i\}$ be a good open cover of $B$ such that $U_{ij} \cnec U_i \cap U_j \subset B^\st$ and that $f_{|X_i} : X_i \cnec f^{-1}(U_i) \to U_i$ has a section $\gs_i : U_i \to X_i$ for every $i$.  We assume that the open cover $\{U_i\}_{i \in I}$ is $G$-invariant and let $G$ act on $I$ by $U_{gi} = g^{-1}(U_i)$. Let $p: \bar{J} \to B$ be the Jacobian fibration associated to $f$ and fix isomorphisms $\eta_i : X_i \to \bar{J}_i$ over $U_i$ sending $\gs_i$ to the 0-section. By construction, $\eta$ is represented by the 1-cocycle $\eta_{ij} =  \eta_i(\gs_i(U_{ij})) - \eta_j(\gs_j(U_{ij}))$.

The $G$-action on $f$ induces a natural $G$-action on $J \to B^\st$ and on $\bar{\cJ}$. Since the $G$-action on $J \to B^\st$ preserves the 0-section, it extends to a meromorphic $G$-action on $\bar{J} \to B$ by~\cite[Proposition 1.6]{NakayamaWeierCnUniv}. For every $g \in G$, let $\psi_g : X \to X$ and $\phi_g : \bar{J} \dto \bar{J}$ denote the action of $g$ on $X$ and on $\bar{J}$ respectively. Then
$$\eta_i \circ \psi_g \circ \eta_{gi}^{-1} \circ \phi_g^{-1} : \bar{J}_i \dto \bar{J}_i \cnec p^{-1}(U_i)$$
is the translation by some section $\eta_i^g$ of $\bar{J}_i \to U_i$ and the 1-cocycle $\{\eta_{ij}\}$ satisfies 
\begin{equation}\label{eq-1coyGcob}
\eta_{ij} - (g^*\eta)_{ij} = \eta_i^g - \eta_j^g.
\end{equation}
where $(g^*\eta)_{ij} \cnec g \cdot \eta_{(gi)(gj)} \in \bar{\cJ}(U_{ij})$, which represents $g^*\eta$. Summing up~\eqref{eq-1coyGcob} over $g \in G$, we have 
$$|G|\cdot \eta_{ij} - \sum_{g\in G}(g^*\eta)_{ij} = \eta^G_i -  \eta_j^G,$$ 
where $\eta^G_i = \sum_{g \in G} \eta_i^g$, which is a section of $\bar{J}_i \to U_i$. So if $X_1 \to B$ and $X_2 \to B$ are the bimeromorphic $J$-torsors twisted by $|G|\cdot \eta_{ij}$ and $\sum_{g\in G}(g^*\eta)_{ij}$ respectively, then the translations $\tr(\eta_i^G) : \bar{J}_i \dto \bar{J}_i$ glue to a bimeromorphic map $X_1 \dto X_2$ over $B$, which proves Lemma~\ref{lem-relevtorG}.
\end{proof}

Finally, we construct the tautological family associated to a $G$-equivariant bimeromorphic $J$-torsor.

\begin{pro-def}\label{pro-def-existfam}
Let $G$ be a finite group and $f : X \to B$ a $G$-equivariant bimeromorphic $J$-torsor. There exists a family of bimeromorphic $J$-torsors
$$ \Pi :  \cX \xto{q}  B \times V \to V \cnec H^1(B, \bar{\cE})^G $$
parameterized by $V$ which satisfies the following properties. 
\begin{enumerate}[label = \roman{enumi})]
\item The family contains $f : X \to B$ as the central fiber and for all $t \in V$, the bimeromorphic $J$-torsor $\cX_t \to B$ parameterized by $t$ corresponds to 
$$\Phi\(\eta(f) + \exp(t)\) \in \cE(B,\gD,\bH)$$
where $\eta(f) \in H^1(B,\bar{\cJ})$ is one (hence any) lift of $[f] \in \cE(B,\gD,\bH)$ by $\Phi$ 
(see Lemma~\ref{lem-surjPhi})
and $\exp : H^1(B, \bar{\cE}) \to H^1(B,\bar{\cJ})$ is the map induced by $\exp : \bar{\cE} \to \bar{\cJ}$ in~\eqref{SE-genJac}.
\item The family $\Pi$ extends the $G$-action on $f$ and is $G$-equivariantly locally trivial over $B$.
\end{enumerate}
The family $\Pi$ is called the \emph{$G$-equivariant tautological family associated to $f$}.
\end{pro-def}
 
\begin{proof}
The construction of the tautological family is similar to the one constructed for smooth torus fibrations. 
By abuse of notation, let $\pr_1$ denote both the first projections $B \times V \to B$ and $B^\st \times V \to B^\st$. Let 
$$\xi \in H^1(B, \bar{\cE}) \otimes H^0(V,\cO_V) \simeq H^1\(B \times V, \pr_1^*\bar{\cE}\) $$ 
be the element corresponding to the inclusion $V \hto H^1(B, \bar{\cE})$. 

 Let $\fU \cnec \{U_i\}_{i \in I}$ be a $G$-invariant good Stein open cover of $B$ such that $U_{ij} \subset B^\st$ for every $i \ne j$ and let $G$ act on $I$ such that $g^{-1}(U_i) = U_{gi}$.  

 \begin{lem}\label{lem-repG0}
 	With respect to the good Stein open cover $\fU \times V = \{U_i \times V\}_{i \in I}$ of $B \times V$, 
 	the element $\xi$ is represented by a $G$-invariant 1-cocycle $\left\{\xi_{ij} : V \to  \bar{\cE}(U_{ij})\right\}$ such that $\xi_{ij}(0) = 0$ (or equivalently, ${\xi_{ij}}_{|U_{ij} \times \{0\}} = 0$). 
 \end{lem}
 \begin{proof}
 	
 	We have the commutative diagram
 	\begin{equation}
 	\begin{tikzcd}[cramped, row sep = 15, column sep = 20]
 	0 \ar[r] & B^1\(\fU \times V, \pr_1^*\bar{\cE} \)^G   \ar[d, "\pi_1"] \ar[r] & Z^1\(\fU \times V, \pr_1^*\bar{\cE} \)^G   \arrow[r] \ar[d, "\pi_2"] & H^1\(B \times V, \pr_1^*\bar{\cE} \)^G \ar[d, "\pi_3"] \ar[r] & 0  \\
 	0 \ar[r] & B^1\(\fU, \bar{\cE} \)^G   \ar[r] & Z^1\(\fU, \bar{\cE} \)^G   \  \arrow[r]  & H^1\(B, \bar{\cE} \)^G  \ar[r] & 0 
 	\end{tikzcd}
 	\end{equation}
 	where $Z^1$ (resp. $B^1$) is the space of \v{C}ech $1$-cocycles (resp. $1$-coboundaries) and
 	the vertical arrows are defined by the restriction to $ B = B \times 0 \subset B \times V$.
 	The rows are exact as the functor $(\bullet)^G$ is left exact and we have
 	$H^1(G, B^1(\fU,\bar{\cE})) = 0$ 
 	(indeed, $H^1(G, B^1(\fU,\bar{\cE}))$ is a $\bC$-vector space and
 	$H^1(G, B^j(\fU,\cF))$ is torsion because $G$ is finite~\cite[Corollary III.10.2]{BrownCohGp})
 	and $H^1(G, B^1(\fU \times V, \pr_1^*\bar{\cE} )) = 0$ 
 	for the same reason. 
 	As $\pi_1$ is surjective, the induced map $\ker(\pi_2) \to \ker(\pi_3)$ is surjective by the snake lemma. Therefore since $\xi_{|B \times \{0\}} = 0 \in H^1\(B, \bar{\cE} \)^G$ by definition of $\xi$, $\xi$ can be represented by a 1-cocycle in $\ker(\pi_2)$, which proves Lemma~\ref{lem-repG0}.
 \end{proof}


Let $p : \bar{J} \to B$ be the Jacobian fibration associated to $f$. 
Recall from~\eqref{eqn-etai'} that we have
biholomorphic maps $\eta_i :  X_i \cnec f^{-1}(U_i)\to \bar{J}_i$
such that $\eta_i \circ \eta_j^{-1} = \tr(\eta_{ij})$ for some \v{C}ech 1-cocycle
$\eta(f) = \{\eta_{ij}\}$ with coefficients in $\bar{\cJ}$.
By construction, we have $\Phi(\eta(f)) = [f] \in \cE(B,\gD,\bH)$. 
 Let $\ti{\eta}_{ij} \in \cJ_{\pr_1^{-1}\bH/B \times V}(U_{ij} \times V)$ be the pullback of $\eta_{ij}$ under $U_{ij} \times V \to U_{ij}$. 
 Let $\left\{\xi_{ij} : V \to \bar{\cE}(U_{ij})\right\}$ be a $G$-invariant 1-cocycle representing $\xi$ as in Lemma~\ref{lem-repG0}. 
 Finally, let
 $$\gl_{ij} \cnec \ti{\eta}_{ij} + \exp(U_{ij} \times V)(\xi_{ij}) \in \cJ_{\pr_1^{-1}\bH/B \times V}(U_{ij} \times V)$$
 where $\exp : \cE_{\pr_1^{-1}\bH/B \times V} \to \cJ_{\pr_1^{-1}\bH/B \times V}$ is the exponential map. 
 
 Let $\bar{J}_i = p^{-1}(U_i)$ and $\bar{J}_{ij} = p^{-1}(U_{ij})$.
  We define  $q : \cX \to B \times V$ to be the fibration obtained by gluing the $\bar{J}_i \times V \to U_i \times V$ along $\bar{J}_{ij} \times V \to U_{ij} \times V$ \emph{via} 
  the 1-cocycle of translations $\tr(\gl_{ij}) : \bar{J}_{ij} \times V \to \bar{J}_{ij} \times V$. By construction, the central fiber of $\Pi$ is $f : X \to B$ and for all $t \in V$, the bimeromorphic $J$-torsor parameterized by $t$ represents $\eta(f) + \exp(t) \in H^1(B,\bar{\cJ})$, which proves $i)$. 


 To construct the $G$-action on $\cX$ and prove $ii)$, we fix  biholomorphic maps $\gl_i : X_i \times V \to \bar{J}_i \times V $  
such that $\gl_i \circ \gl_j^{-1} = \tr(\gl_{ij})$. For every $g \in G$, let $\psi_g^i : X_{gi} \to X_i$ be the restriction to $X_{gi}$ of the action of $g$ on $X$. Define
$$\Psi_g^i  \cnec \gl_i^{-1} \circ \((\eta_i  \circ \psi_g^i \circ \eta_{gi}^{-1}) \times \Id_V\) \circ \gl_{gi} : X_{gi} \times V \to X_i \times V.$$
Then we can glue the $\Psi_g^i$ together and obtain a biholomorphic map $\Psi_g : \cX \to \cX$ such that $g \mapsto \Psi_g$ is a $G$-action on $\cX$ extending the $G$-action on $X$. With this $G$-action on $\cX$, $q$ is $G$-equivariant. It follows from the construction that $\Pi$ is $G$-equivariantly locally trivial over $B$, which proves $ii)$.
\end{proof}%

\section{The tautological family is an algebraic approximation}\label{sec-appalgtau}

We continue to use the notations introduced in Section~\ref{sec-genjac}. Now we prove that the tautological family constructed at the end of Section~\ref{sec-genjac} is an algebraic approximation.

\begin{proof}[Proof of Theorem~\ref{thm-AbFibDefprec}] 
Let $B^\st \subset B$ be the Zariski open subset parameterizing smooth fibers of $f$ and let $J \to B^\st$ be the Jacobian fibration associated to ${f}_{|X^\st} : {X}^\st \cnec f^{-1}(B^\st) \to B^\st$. 
Then $f$ is a $G$-equivariant bimeromorphic $J$-torsor. 
Let $\eta(f) \in H^1(B,\bar{\cJ})$ be a lifting of $[f] \in  \cE(B,\gD,\bH)$ under $\Phi$ (see Lemma~\ref{lem-surjPhi}). 
By Proposition-Definition~\ref{pro-def-existfam}, the $G$-equivariant tautological family 
$$\Pi : {\cX} \to B \times V \to V \cnec  H^1(B, \bar{\cE})^G$$ 
 associated to $f : X \to B$ is a deformation of $f$
 which is $G$-equivariantly locally trivial over $B$,  and each $t \in V$
 parameterizes a bimeromorphic $J$-torsor $f_t : \cX_t \to B$ such that $\Phi\(\eta(f) + \exp(t)\) = [f_t]$. Since the fibers of $f_t$ are projective, by Corollary~\ref{cor-multsecMoibase} it suffices to show that there exists a dense subset $V_\tors \subset V$ parameterizing bimeromorphic $J$-torsors with multi-sections in the tautological family.
 
 For every $\xi \in H^1(B,\bar{\cJ})$, let $f^\xi : X^\xi \to B$ be a bimeromorphic $J$-torsor such that $\Phi(\xi) = [f^\xi] \in \cE(B,\gD,\bH)$. 
 By Proposition~\ref{pro-relevtor}, there exists  $m \in \bZ_{>0}$ such that $[f_m] \in \cE(B,\gD,\bH)$ can be lifted to some element $\eta(f_m) \in H^1(B,\bar{\cJ})$
 satisfying $c(\eta(f_m)) = 0$.
 By Lemma-Definition~\ref{lemdef-multm}, 
 we have $\Phi(m \cdot\eta(f)) = \Phi(\eta(f_m))$. 
 Let $\bar{\eta} \cnec \sum_{g \in G} g^*\eta(f_m)$. For each $t \in V$, there exists a generically finite map
\begin{equation}
\begin{tikzcd}[cramped, row sep = 2.5]
\cX_t \arrow[r, dashed, "\bm"] & X^{m \cdot \(\eta(f) + \exp(t)\)}  \arrow[r,dashed, "\sim"] & 
X^{ \(\eta(f_m) + m\exp(t)\)} \arrow[r,dashed, "\mathbf{|G|}"] & X^{|G| \cdot \(\eta(f_m) + m\exp(t)\)} 
\arrow[r,dashed, "\sim"] & X^{\bar{\eta} + \exp(m \cdot |G| \cdot t)}
\end{tikzcd}
\end{equation}
 over $B$, where the first (resp. third) map is the multiplication by $m$ 
 (resp. by $|G|$), the second map is the bimeromorphic map which follows from
 $\Phi(m \cdot \eta(f)) = \Phi(\eta(f_m))$  
 and the fourth one is the bimeromorphic map given by Lemma~\ref{lem-relevtorG}. So according to Lemma~\ref{lem-torsmulsc}, it suffices to show that the subset 
 $$V_\tors \cnec \left\{t \in V \mid \bar{\eta} + \exp(m \cdot |G| \cdot t) \in H^1(B,\bar{\cJ})_\tors \right\}$$
 is dense in $V$.

As $c(\eta(f_m)) = 0$  and the map $c$ is $G$-equivariant, by the exact sequence
\begin{equation}\label{se-lhej}
\begin{tikzcd}[cramped]
H^1(B,j_*\bH) \ar[r,"\phi_\bZ"] & H^1(B, \bar{\cE}) \ar[r,"\exp"]  & H^1({B}, \bar{\cJ})  \ar[r,"c"] & H^2(B, j_*\bH) 
\end{tikzcd}
\end{equation}
induced by~\eqref{SE-genJac}, there exists $\gb \in H^1(B,\bar{\cE})^G = V$ such that $\exp(\gb) =  \bar{\eta}$. 
Let $\phi: H^1({B},j_*\bH)_\bQ \to H^1({B}, \bar{\cE})$ be the $\bQ$-tensorization of $\phi_\bZ$ and let $W \colonec H^1({B},j_*\bH)^G_\bQ$. Since $G$ is finite, we have $\Ima(\phi)^G = \phi(W)$. By Corollary~\ref{Cor-Zuck}, we have $\Ima(\phi) \otimes \bR = H^1({B}, \bar{\cE})$, so 
$$\phi(W) \otimes \bR = \Ima(\phi)^G \otimes \bR = H^1({B}, \bar{\cE})^G = V.$$ 
In particular since $\phi(W)$ is a $\bQ$-vector space, $\phi(W)$ is dense in $V$, so $\phi(W) - \gb$ is dense in $V$. 

By~\eqref{se-lhej}, we have $\exp(\phi(W)) \subset H^1(B,\bar{\cJ})_\tors$, 
so $\phi(W) - \frac{\gb}{m\cdot |G|} \subset V_\tors$. Therefore $V_\tors$ is dense in $V$.
\end{proof}

\begin{proof}[Proof of Corollary~\ref{cor-AbFibDef}]

First we show that to prove Corollary~\ref{cor-AbFibDef},
we can freely modify $f : X \to B$ along fibers of $f$.

\begin{claim}\label{claim-mod}
	Let $\gS \subset B$ be a finite subset and 
	let $f' : X' \to B$ be a bimeromorphic modification of $f$ 
	such that the underlying map $\mu : X' \dto X$ is an isomorphism over $B \bss \gS$.
	If $f'$ has an algebraic approximation which is locally trivial over $B$,
	then $f$ also has an algebraic approximation which is locally trivial over $B$.
\end{claim}
\begin{proof}
	Let $\cX' \xto{q'} B \times \gD \to \gD$ be an algebraic approximation of $f'$ 
	which is locally trivial over $B$.
	Let $o \in \gD$ be the point parameterizing $f'$
	(namely $q'^{-1}(B \times \{o\}) \to B \times \{o\}$ is  $f' : X' \to B$). 
	Since the deformation is locally trivial over $B$, 
	for every $p \in \gS$, there exist a neighborhood 
	$U_p$ of $p$ and an isomorphism
	$$\phi_p : \cU'_p \cnec q'^{-1}(U_p \times \gD) \simeq f'^{-1}(U_p) \times \gD$$
	over $U_p \times \gD$.
	We can assume that the restriction of $\phi_p$ to 
	$q'^{-1}(U_p \times \{o\}) = f'^{-1}(U_p) \times \{o\}$ is the identity.
	Up to shrinking $U_p$, we can assume that $U_p \cap U_q = \emptyset$ whenever $p \ne q \in \gS$.
	Let 
	$$\mu_p : \cU'_p \simeq f'^{-1}(U_p) \times \gD \dto f^{-1}(U_p) \times \gD =: \cU_p$$
	where the dashed arrow is the bimeromorphic map $\mu_{|f'^{-1}(U_p)} \times \Id_{\gD}$ 
	(over $U_p \times \gD$).
	Finally, let $\cU \cnec q'^{-1}((B \bss \gS) \times \gD)$.
	
	Since each ${\mu_p}_{|\cU'_p \cap \cU}$ is an isomorphism onto its image
	and since $\cU'_p \cap \cU_q' = \emptyset$ whenever $p \ne q$,
	we can glue the bimeromorphic maps 
	$\mu_p : \cU'_p \dto \cU_p$ (over $U_p \times \gD$)
	and $\Id_\cU : \cU \to \cU$ (over $(B \bss \gS) \times \gD$) and form a bimeromorphic map
	$\ti{\mu} : \cX' \dto \cX$ over $B \times \gD$.
	Let $q : \cX \to B \times \gD$ denote the structural map.
	By construction, over $B \times \{o\}$ the restriction
	$q'^{-1}(B \times \{o\}) \dto q^{-1}(B \times \{o\})$ of $\ti{\mu}$ is isomorphic (over $B$) to 
	$\mu : X' \dto X$.
	
	Since 
	$$ q^{-1}(U_p \times \gD) = \cU_p =  f^{-1}(U_p) \times \gD \xto{q_{|\cU_p}} U_p \times \gD \to \gD$$ 
	is clearly locally trivial over $U_p$ for every $p \in \gS$ by construction and since
	$$ q^{-1}((B\bss \gS) \times \gD) = \cU \xto{q_{|\cU} = q'_{|\cU}} (B\bss \gS) \times \gD \to \gD$$ 
	is locally trivial over $B\bss \gS$ by assumption,
	it follows that their gluing 
	$\cX \xto{q} B \times \gD \to \gD$ 
	is locally trivial over $B$
	as a deformation of $f : X \to B$.
	Finally, as $\cX' \dto \cX$ is bimeromorphic over $\gD$ and 
	$\cX' \to \gD$ is an algebraic approximation of $f'$, 
	necessarily $\cX \to \gD$ is an algebraic approximation $f$.
\end{proof}

Let $\{U_i\}$ be a good open cover of $B$. Since the fibers of $f : X \to B$ are algebraic, by~\cite[Corollaire du Théorème 2]{Campana-redalg} the restriction of $f$ to $f^{-1}(U_i)$ is Moishezon. Thus for each $i$, there exists a connected multi-section $Z_i \subset X_i  \cnec f^{-1}(U_i)$ of $f_{|X_i}$. Up to refining the open cover, we can assume that each $U_i$ is a disc and the multi-section $Z_i$  is \'etale over $ U_i - \{o_i\}$ where $o_i$ is the center of $U_i$. 
We can also assume that $o_j \in U_i$ if and only if $i = j$.


Let $\gS \colonec \{o_i \mid i \in I\}  \subset B$ be the set of centers of $U_i$. Let $d_i$ denote the degree of $Z_i \to U_i$ and $d \colonec \lcm \{d_i\}$. Up to adding more points to $\gS$, the fundamental group $\pi_1(B \bss \gS)$ is free, so there exists a cyclic cover $r : \ti{B} \to B$ of degree $d$ branched along $\gS$ by Riemann's extension theorem; let $G \cnec \Gal(\ti{B} / B )$ denote the corresponding Galois group.
Let $\ti{f} : \ti{X} \to \ti{B}$ denote the base change of $f : X \to B$ by $\ti{B} \to B$. 
By construction, the fibration $\ti{f}$ is $G$-equivariant and a general fiber of $\ti{f}$ is still an abelian variety. 
Also, the pre-image of the multi-section $Z_i \subset X_i \to U_i$ in $\ti{X}$ contains a local section of $\ti{f} : \ti{X} \to \ti{B}$ for each $i$. 
Finally, let $\nu : \ti{X}' \to \ti{X}$ be 
a strong $G$-equivariant K\"ahler desingularization of $\ti{X}$ (see Theorem~\ref{thm-mindesingG}).

By construction, $\ti{f}' : \ti{X}' \xto{\nu} \ti{X} \xto{\ti{f}} \ti{B}$ 
is a $G$-equivariant fibration satisfying the hypotheses of Theorem~\ref{thm-AbFibDefprec}.
By Theorem~\ref{thm-AbFibDefprec}, there exists an algebraic approximation
$$\Pi : \ti{\cX}' \to \ti{B} \times V \to V$$
 of $\ti{f}'$ which is $G$-equivariantly locally trivial over $\ti{B}$. 
 So the quotient of $\Pi$ by $G$ is an algebraic approximation of $\ti{X}'/G \to B$, 
 which is locally trivial over $B$ by Lemma~\ref{lem-Gquotloctriv}.

Finally, since $\ti{f} : \ti{X} \to \ti{B}$ is smooth over $\ti{B} \bss r^{-1}(\gS \cup \gS')$ 
where $\gS' \subset B$ is the finite subset 
parameterizing singular fibers of $f : X \to B$,
and since $\nu : \ti{X}' \to \ti{X}$ is a strong K\"ahler desingularization of $\ti{X}$,
the bimeromorphic modification $\nu :  \ti{X}' \dto \ti{X}$ is an isomorphism over 
$\ti{B} \bss r^{-1}(\gS \cup \gS')$.
So the quotient $\ti{X}'/G$ is isomorphic to $X$ over $B \bss (\gS \cup \gS')$. 
Hence by Claim~\ref{claim-mod}, 
$f:X \to B$ has an algebraic approximation
which is locally trivial over $B$.
\end{proof}%

\section{Appendix: A lemma about the Grothendieck spectral sequences}

Let $F : \cA \to \cB$  be a left exact functor of small abelian categories. We assume that $\cA$ has enough injectives. The following result is presumably well-known, yet a proof is difficult to find in the literature.

\begin{lem}\label{lem-appenComm}
Let 
\begin{equation}\label{diag-LMNexact}
\begin{tikzcd}[cramped, row sep = 2.5]
0 \arrow[r] & L^\bullet \arrow[r] & M^\bullet \arrow[r] & N^\bullet \arrow[r] & 0  
\end{tikzcd}
\end{equation}
be a short exact sequence of bounded complexes in $\cA$, which induces the long exact sequence
\begin{equation}
\begin{tikzcd}[cramped, row sep = 2.5]
\cdots  \arrow[r] & H^{q-1}(L^\bullet)  \arrow[r,"\phi"] & H^{q-1}(M^\bullet)  \arrow[r] & H^{q-1}(N^\bullet) \arrow[r, "\rho"] & H^{q}(L^\bullet)   \arrow[r] & H^{q}(M^\bullet)  \arrow[r,"\psi"] &  H^{q}(N^\bullet)  \arrow[r] & \cdots 
\end{tikzcd}
\end{equation}
Then 
\begin{equation}\label{diag-lemtechcom}
\begin{tikzcd}[cramped, row sep = 20]
R^pF(H^q(M^\bullet))  \arrow[rr,"d_2"] & & R^{p+2}F(H^{q-1}(M^\bullet)) \arrow[d]  \\
R^pF(\ker(\psi)) \arrow[u, "h"]   \arrow[r, "f"] & R^{p+1}F(\Ima(\rho))   \arrow[r,"g"] & R^{p+2}F(\coker(\phi) )  
\end{tikzcd}
\end{equation}
is commutative where the morphism in the top row is the $d_2$ map of the spectral sequence
$$E_2^{p,q} = R^pF(H^q(M^\bullet)) \Rightarrow R^{p+q}F(M^\bullet),$$
and the bottom row are the connecting morphisms induced by the exact sequences 
\begin{equation}\label{SE-appen1}
\begin{tikzcd}[cramped, row sep = 2.5]
0 \arrow[r] & \coker(\phi) \arrow[r] & H^{q-1}(N^\bullet) \arrow[r] & \Ima(\rho) \arrow[r] & 0 
\end{tikzcd}
\end{equation}
\begin{equation}\label{SE-appen2}
\begin{tikzcd}[cramped, row sep = 2.5]
0 \arrow[r] & \Ima(\rho) \arrow[r] & H^{q}(L^\bullet) \arrow[r] & \ker(\psi) \arrow[r] & 0.
\end{tikzcd}
\end{equation}
 
\end{lem} 

\begin{proof}

Let 
\begin{equation}\label{diag-CEexact}
\begin{tikzcd}[cramped, row sep = 2.5]
0 \arrow[r] & (L^{\bullet,\bullet}, \delta_1,\delta_2) \arrow[r] & (M^{\bullet,\bullet}, \delta_1,\delta_2)  \arrow[r] & (N^{\bullet,\bullet}, \delta_1,\delta_2) \arrow[r] & 0  
\end{tikzcd}
\end{equation}
be an exact sequence of Cartan-Eilenberg resolutions extending~\eqref{diag-LMNexact}. We define $I^{\bullet,\bullet} \colonec F(L^{\bullet,\bullet})$, $J^{\bullet,\bullet} \colonec F(M^{\bullet,\bullet})$ and $K^{\bullet,\bullet} \colonec F(N^{\bullet,\bullet})$. Since there will be no ambiguity, let $\delta_i$ also denote $F(\delta_i)$.
We set 
$$\ker(\psi)^\bullet \colonec \ker(H^q(M^{\bullet,\bullet}, \delta_1) \to H^q(N^{\bullet,\bullet}, \delta_1)), $$ 
$$\Ima(\delta)^\bullet \colonec \Ima(H^{q-1}(N^{\bullet,\bullet}, \delta_1) \to H^q(L^{\bullet,\bullet}, \delta_1)), $$
$$\coker(\phi)^\bullet \colonec \coker(H^{q-1}(L^{\bullet,\bullet}, \delta_1) \to H^{q-1}(M^{\bullet,\bullet}, \delta_1)), $$
 which are injective resolutions of $\ker(\psi)$, $\Ima(\delta)$ and $\coker(\phi)$ respectively.

The map $g \circ f$ is constructed as follows.
Let $\ga \in R^pF(\ker(\psi))$ which is represented by some element $\ga_0 \in F(\ker(\psi)^p)$. By the right-exactness of~\eqref{SE-appen2}, $\ga_0$ is further represented by some element $\tilde{\ga}$ in the image of the monomorphism $I^{q,p} \to J^{q,p}$. By definition, $f(\ga)$ is 
 represented by some element in $F(\Ima(\rho)^{p+1})$ which is further represented by $\delta_2^{q,p}(\tilde{\ga})$ as an element in $I^{q,p+1}$. Now by the right-exactness of~\eqref{SE-appen1}, there exists $\gb_0 \in K^{q-1,p+1}$ which represents a pre-image of the class of $\delta^{q,p}_2(\tilde{\ga})$ in $H^{q}(I^{\bullet,p+1},\delta_1)$ under the morphism $H^{q-1}(K^{\bullet,p+1},\delta_1) \to H^{q}(I^{\bullet,p+1},\delta_1)$. Since $J^{q-1,p+1} \to K^{q-1,p+1}$ is surjective, we can find $\gb \in J^{q-1,p+1}$ mapping to $\gb_0 \in K^{q-1,p+1}$. Again by definition, the class of $ \delta^{q-1,p+1}_2(\gb)$ in $F(\coker(\phi)^{p+2})$ is an element representing $(g\circ f)(\ga)$.
 
 Next we recall the construction of the map $d_2$. Let 
$$\ga \in R^pF(H^q(M^\bullet)) \simeq H^{p}\(H^q(J^{q,p}, \delta_1),\delta_2\)$$ 
which is represented by $\tilde{\alpha} \in  J^{q,p} $. Since $\delta_2^{q,p}(\tilde{\alpha}) = 0$ in $H^q(J^{\bullet,p+1}, \delta_1)$ there exists $\gb \in J^{q-1,p+1}$ such that 
$$\delta_1^{q-1,p+1}(\gb) = \delta_2^{q,p}(\tilde{\alpha}) \in J^{q,p+1}. $$ 
Let $\gb' \colonec \delta_2^{q-1,p+1}(\gb)$. Note that 
$$\delta_1^{q-1,p+2}(\gb') = \delta_2^{q,p+1}\delta_1^{q-1,p+1}(\gb) = \delta_2^{q,p+1}\delta_2^{q,p}(\tilde{\ga}) = 0,$$ we have $\gb' \in \ker \delta_1^{q-1,p+2}$. Let $\tilde{\gb}$ be the class of $\gb'$ in $H^{q-1}(K^{\bullet,p+2}, \delta_1)$. Since 
$$ \delta_2^{q-1,p+2} (\gb') =  \delta_2^{q-1,p+2} \delta_2^{q-1,p+1}(\gb) = 0,$$
we have 
$$\tilde{\gb} \in \ker\(\delta_2 : H^{q-1}(J^{\bullet,p+2}, \delta_1) \to H^{q-1}(J^{\bullet,p+3}, \delta_1)\),$$ and $d_2(\ga)$ is defined to be the class of $\tilde{\gb}$ in $R^{p-1}F(H^{q+2}(M^\bullet))\simeq H^{p-1}\(H^{q+2}(J^{q+2,p-1}, \delta_1),\delta_2\)$. 

The construction of $d_2(\ga)$ is independent of the choices of $\tilde{\ga}$ and $\gb$. If $\ga = h(\ga')$ for some $\ga \in R^pF(\ker(\psi))$, then we can choose the same $\tilde{\ga}$ and $\gb$ as in the construction of $(g\circ f)$. Therefore $(g\circ f)(\ga')$ and $(d_2 \circ h)(\ga')$ are both represented by $\delta^{q-1,p+1}_2(\gb)$, which proves that~\eqref{diag-lemtechcom} is commutative.
\end{proof}


%

\section*{Acknowledgement}
This work was carried out while the author was first supported by the SFB 647: "Raum - Zeit - Materie. Analytische und Geometrische Strukturen" at the Humboldt University in Berlin then by the SFB/TR 45 "Periods, Moduli Spaces and Arithmetic of Algebraic Varieties" of the DFG (German Research Foundation) at the University of Bonn.  I am grateful to F. Campana, B. Claudon, and F. Gounelas for general discussions and e-mail correspondences related to this work. 
Last but not least, I would like to thank the referee for carefully reading the article and for
her/his valuable questions and suggestions.

\bibliographystyle{plain}
\bibliography{FibAb}

\begin{thebibliography}{10}

\bibitem{BrownCohGp}
Kenneth~S. Brown.
\newblock {\em Cohomology of groups}, volume~87 of {\em Graduate Texts in
  Mathematics}.
\newblock Springer-Verlag, New York, 1994.
\newblock Corrected reprint of the 1982 original.

\bibitem{Campana-redalg}
Fr\'ed\'eric Campana.
\newblock R\'eduction alg\'ebrique d'un morphisme faiblement {K}\"ahl\'erien
  propre et applications.
\newblock {\em Math. Ann.}, 256(2):157--189, 1981.

\bibitem{CampanaPeternell2-form}
Fr\'ed\'eric Campana and Thomas Peternell.
\newblock Complex threefolds with non-trivial holomorphic {$2$}-forms.
\newblock {\em J. Algebraic Geom.}, 9(2):223--264, 2000.

\bibitem{CampanaCored}
Frédéric Campana.
\newblock Cor\'eduction alg\'ebrique d'un espace analytique faiblement
  k\"ahl\'erien compact.
\newblock {\em Invent. Math.}, 63(2):187--223, 1981.

\bibitem{CaoJApproxalg}
Junyan Cao.
\newblock On the approximation of {K}\"ahler manifolds by algebraic varieties.
\newblock {\em Math. Ann.}, 363(1-2):393--422, 2015.

\bibitem{ClaudonToridefequiv}
Beno\^{\i}t Claudon.
\newblock Smooth families of tori and linear {K}\"{a}hler groups.
\newblock {\em Ann. Fac. Sci. Toulouse Math. (6)}, 27(3):477--496, 2018.

\bibitem{ClaudonHorpi1}
Beno\^{\i}t Claudon, Andreas H\"{o}ring, and Hsueh-Yung Lin.
\newblock The fundamental group of compact {K}\"{a}hler threefolds.
\newblock {\em Geom. Topol.}, 23(7):3233--3271, 2019.

\bibitem{ElZeinZuck}
Fouad El~Zein and Steven Zucker.
\newblock Extendability of normal functions associated to algebraic cycles.
\newblock In {\em Topics in transcendental algebraic geometry ({P}rinceton,
  {N}.{J}., 1981/1982)}, volume 106 of {\em Ann. of Math. Stud.}, pages
  269--288. Princeton Univ. Press, Princeton, NJ, 1984.

\bibitem{FujikiClosednessDouady}
Akira Fujiki.
\newblock Closedness of the {D}ouady spaces of compact {K}\"{a}hler spaces.
\newblock {\em Publ. Res. Inst. Math. Sci.}, 14(1):1--52, 1978/79.

\bibitem{FujikiStruC}
Akira Fujiki.
\newblock On the structure of compact complex manifolds in {${\mathcal C}$}.
\newblock In {\em Algebraic varieties and analytic varieties ({T}okyo, 1981)},
  volume~1 of {\em Adv. Stud. Pure Math.}, pages 231--302. North-Holland,
  Amsterdam, 1983.

\bibitem{GrafDefKod0}
Patrick Graf.
\newblock Algebraic approximation of {K}{\"a}hler threefolds of {K}odaira
  dimension zero.
\newblock {\em Mathematische Annalen}, Aug 2017.

\bibitem{KodairaSurfaceII}
Kunihiko Kodaira.
\newblock On compact analytic surfaces. {II}, {III}.
\newblock {\em Ann. of Math. (2) 77 (1963), 563--626; ibid.}, 78:1--40, 1963.

\bibitem{resing}
J\'{a}nos Koll\'{a}r.
\newblock {\em Lectures on resolution of singularities}, volume 166 of {\em
  Annals of Mathematics Studies}.
\newblock Princeton University Press, Princeton, NJ, 2007.

\bibitem{HYLkod3}
Hsueh-Yung {Lin}.
\newblock {Algebraic approximations of compact Kähler threefolds}.
\newblock {\em ArXiv:1710.01083}, 2018.

\bibitem{HYLkodfibellip}
Hsueh-Yung Lin.
\newblock Algebraic approximations of compact {K}\"{a}hler manifolds of
  algebraic codimension 1.
\newblock {\em Duke Math. J.}, 169(14):2767--2826, 2020.

\bibitem{NakayamaWeierCnUniv}
Noboru Nakayama.
\newblock Compact {K}ähler manifolds whose universal covering spaces are
  biholomorphic to $\mathbf{C}^n$.
\newblock {\em Preprint RIMS-\em{1230}, Res. Inst. Math. Sci. Kyoto Univ.},
  1999.

\bibitem{PetersSteenbMHS}
Chris A.~M. Peters and Joseph H.~M. Steenbrink.
\newblock {\em Mixed {H}odge structures}, volume~52 of {\em Ergebnisse der
  Mathematik und ihrer Grenzgebiete. 3. Folge. A Series of Modern Surveys in
  Mathematics [Results in Mathematics and Related Areas. 3rd Series. A Series
  of Modern Surveys in Mathematics]}.
\newblock Springer-Verlag, Berlin, 2008.

\bibitem{MSaitoAdmnorm}
Morihiko Saito.
\newblock Admissible normal functions.
\newblock {\em J. Algebraic Geom.}, 5(2):235--276, 1996.

\bibitem{Schrackdefo}
Florian Schrack.
\newblock Algebraic approximation of {K}\"ahler threefolds.
\newblock {\em Math. Nachr.}, 285(11-12):1486--1499, 2012.

\bibitem{VarouchasKS}
Jean Varouchas.
\newblock K\"{a}hler spaces and proper open morphisms.
\newblock {\em Math. Ann.}, 283(1):13--52, 1989.

\bibitem{VoisinII}
Claire Voisin.
\newblock {\em Hodge Theory and Complex Algebraic Geometry II}, volume~77 of
  {\em Cambridge Studies in Advanced Mathematics}.
\newblock Cambridge University Press, 2003.

\bibitem{Voisincs}
Claire Voisin.
\newblock On the homotopy types of compact {K}ähler and complex projective
  manifolds.
\newblock {\em Invent. Math.}, 157:329--343, 2004.

\bibitem{MR2500573}
Jaros{\l}aw W{\l}odarczyk.
\newblock Resolution of singularities of analytic spaces.
\newblock In {\em Proceedings of {G}\"{o}kova {G}eometry-{T}opology
  {C}onference 2008}, pages 31--63. G\"{o}kova Geometry/Topology Conference
  (GGT), G\"{o}kova, 2009.

\bibitem{ZuckerGenJac}
Steven Zucker.
\newblock Generalized intermediate {J}acobians and the theorem on normal
  functions.
\newblock {\em Invent. Math.}, 33(3):185--222, 1976.

\bibitem{ZuckerHdgL2}
Steven Zucker.
\newblock Hodge theory with degenerating coefficients. {$L_{2}$} cohomology in
  the {P}oincar\'e metric.
\newblock {\em Ann. of Math. (2)}, 109(3):415--476, 1979.

\bibitem{ZuckerDHB}
Steven Zucker.
\newblock Degeneration of {H}odge bundles (after {S}teenbrink).
\newblock In {\em Topics in transcendental algebraic geometry ({P}rinceton,
  {N}.{J}., 1981/1982)}, volume 106 of {\em Ann. of Math. Stud.}, pages
  121--141. Princeton Univ. Press, Princeton, NJ, 1984.

\end{thebibliography}

\end{document}